\documentclass[10pt]{article}
\usepackage{amsthm}
\usepackage[top=1in, bottom=1in, left=1in, right=1in]{geometry}

\usepackage{amsmath,amssymb,amsfonts,amsxtra}
\usepackage{IEEEtrantools}
\usepackage{verbatim}
\usepackage{times}
\usepackage{enumerate}
\usepackage{url}

\newcommand{\ack}{\section*{Acknowledgments}}

\theoremstyle{plain}
\newtheorem{thm}{Theorem}[subsection]
\newtheorem{lem}[thm]{Lemma}

\theoremstyle{definition}
\newtheorem{prob}[thm]{Problem}

\numberwithin{equation}{subsection}

\newcommand{\Comment}[1]{}
\newcommand{\rbr}[1]{\left( {#1} \right)}
\newcommand{\sbr}[1]{\left[ {#1} \right]}
\newcommand{\cbr}[1]{\left\{ {#1} \right\}}

\newcommand{\abs}[1]{\left| {#1} \right|}
\newcommand{\norm}[1]{\left\|#1\right\|}
\newcommand{\eq}[1]{(\ref{#1})}

\newcommand{\ceil}[1]{\left\lceil#1\right\rceil}

\def\pa{Papadimitropoulos }

\def\CC{\mathbb{C}}
\def\RR{\mathbb{R}}
\def\QQ{\mathbb{Q}}
\def\ZZ{\mathbb{Z}}
\def\NN{\mathbb{N}}

\def\supp{\text{supp}}

\def\one{\mathbf{1}}
\def\forinfmany{\text{ for infinitely many }}
\begin{document}

%\title[Explicit Salem sets in the $p$-adic numbers]{Explicit Salem sets, Fourier restriction, and metric Diophantine approximation in the $p$-adic numbers}
\title{Explicit Salem sets, Fourier restriction, and metric Diophantine approximation in the $p$-adic numbers}
\author{Robert Fraser and Kyle Hambrook}
%\date{\today}

%\address{Robert Fraser \\ Department of Mathematics, University of British Columbia, Vancouver, BC V6T1Z2, Canada}
%\curraddr{}
%\email{rgf@math.ubc.ca} 
%\address{Kyle Hambrook \\
%Department of Mathematics, University of Rochester, Rochester, NY 14627, USA}
%\curraddr{}
%\email{khambroo@ur.rochester.edu}
%\thanks{This work was supported by NSERC.}

%\subjclass[2010]{Primary 43A70, 43A25; Secondary 11J83, 11J61, 42A38, 42B10, 28A78, 28A80}
%\keywords{Salem sets, 
%Fourier restriction,  
%metric Diophantine approximation, $p$-adic, Hausdorff dimension, Fourier dimension}
%\dedicatory{}
%\commby{}

\maketitle

\begin{abstract}
We exhibit the first explicit examples of Salem sets in $\mathbb{Q}_p$ of every dimension $0 < \alpha < 1$ by showing 
that certain sets of well-approximable $p$-adic numbers are Salem sets. 
We construct measures supported on these sets that satisfy essentially optimal Fourier decay and upper regularity conditions, 
and we observe that these conditions imply that the measures satisfy strong Fourier restriction inequalities. 
We also partially generalize our results to higher dimensions.  
Our results extend theorems of Kaufman, Papadimitropoulos, and Hambrook from the real to the $p$-adic setting. 
\end{abstract}

\section{Introduction}\label{intro}

\subsection{Basic Notation}

For $x \in \RR^d$ and $r>0$, $|x| = \max_{1 \leq i \leq d} |x_i|$ and $B(x,r) = \cbr{y \in \RR^d : |x-y| \leq r}$.  
Throughout, $p$ denotes a fixed but arbitrary prime number and $\QQ_p$ is the field of $p$-adic numbers. 
The basics of $\QQ_p$ are reviewed in Section \ref{p-adic basics}. 
For $x \in \QQ_p$, $|x|_p$ is the $p$-adic absolute value of $x$.  
For $x \in \QQ_p^d$ and $r>0$, $|x|_p = \max_{1 \leq i \leq d} |x_i|_p$ and $B(x,r) = \cbr{y \in \QQ_p^d : |x-y|_p \leq r}$.  
The ring of integers of $\QQ_p$ is $\ZZ_p = B(0,1)$. 
The Fourier transform of a measure $\mu$ on $\RR^d$ or $\QQ_p^d$ is denoted $\widehat{\mu}$. 
Fourier analysis on $\QQ_p^d$ is reviewed in Section \ref{p-adic Fourier}. 
The expression $X \lesssim Y$ means $X \leq CY$ for some positive constant $C$ whose value may depend on $p$, but not on any other parameters. The expression $X \lesssim_{\alpha} Y$ has the same meaning, except the constant $C$ is permitted to depend also on a parameter $\alpha$. The expression $X \approx Y$ means both $X \lesssim Y$ and $Y \lesssim X$.

\subsection{Salem Sets and Fourier Dimension: The Real Setting}

It is well-known (see, for example, \cite{falconer-book-1}, \cite{mattila-book-1}) that the Hausdorff dimension 
$\dim_H(A)$ of a Borel set $A \subseteq \RR^d$ can 
be expressed 
in terms of the average Fourier decay of measures on $A$: 
\begin{align}\label{haus-form}
\dim_{H}(A) = \sup\cbr{ 0 \leq \alpha \leq d : \int_{\RR^d} |\widehat{\mu}(\xi)|^2 |\xi|^{\alpha-d} d\xi < \infty 
%\text{ for some } 
\quad \exists 
\mu \in \mathcal{P}(A)},
\end{align}
where $\mathcal{P}(A)$ denotes the set of Borel probability measures on $\RR^d$ with compact support in $A$.

The Fourier dimension $\dim_F(A)$ of a set $A \subseteq \RR^d$ is defined in terms of the pointwise Fourier decay of measures on $A$: 
\begin{align*}%\label{fourier-def}
\dim_{F}(A) = \sup\cbr{ 0 \leq \beta \leq d : \sup_{0 \neq \xi \in \RR^d} |\widehat{\mu}(\xi)|^2 |\xi|^{\beta} < \infty 
%\text{ for some } 
\quad \exists 
\mu \in \mathcal{P}(A)}.
\end{align*}

The Fourier dimension of a Borel set is always less than or equal to its Hausdorff dimension.
%For every Borel set $A \subseteq \RR^d$, $\dim_F(A) \leq \dim_H(A)$ 
In general, they are not equal. 
%In general, the dimensions are not the same.
%In $\RR^d$ with $d \geq 2$, every subset of a hyperplane has Fourier dimension 0, but the Hausdorff dimension may be any number between $0$ and $d-1$. 
In $\RR^d$ with $d \geq 2$, subsets of hyperplanes must have Fourier dimension 0, but the Hausdorff dimension may be any number between $0$ and $d-1$.
For $d=1$, the middle-thirds Cantor set in $\RR$ has Fourier dimension 0 and Hausdorff dimension $\ln 2 / \ln 3$.  
%Sets $A \subseteq \RR^d$ with $$\dim_H(A) = \dim_F(A)$$ are called Salem sets.
Some subtle properties of Fourier dimension are studied by Ekstr{\"o}m, Persson, and Schmeling \cite{EPS} and Fraser, Orponen, and Sahlsten \cite{FOS}.

A set whose Fourier and Hausdorff dimensions are equal is called a Salem set.

Points, spheres, and balls in $\RR^d$ 
%%%or $\QQ_p^d$ 
are Salem sets of dimension $0$, $d-1$, and $d$, respectively. 
Salem sets are named for Rapha\"{e}l Salem \cite{salem}, who proved the existence of Salem sets in $\RR$ of every dimension $0 < \alpha < 1$ via a construction of  Cantor sets with random contraction ratios. 
Kahane \cite{kahane-book} proved the existence of Salem sets in $\RR^d$ of every dimension $0 < \alpha < d$ by considering 
trajectories  
of Brownian motion and more general stochastic processes. 
There are many other random constructions of Salem sets in $\RR^d$ (see \cite{bluhm-1}, \cite{chen-seeger}, \cite{ekstrom}, \cite{LP2009}, \cite{shmerkin-suomala}).

The random constructions of Salem sets mentioned above are unsatisfactory in that they do not give \textit{explicit} Salem sets. At best, they give families whose members are (with respect to some measure) almost all Salem sets.

Kaufman \cite{kaufman} was the first to give a construction of explicit Salem sets in $\RR$ of every dimension $0 < \alpha < 1$. His construction comes from number theory and is (arguably) simpler than the random constructions mentioned above.

For $\tau \in \RR$, the set of $\tau$-well-approximable real numbers is 
$$
E(\tau) = \cbr{x \in [-1,1] : |qx-r| \leq \max(|q|,|r|)^{-\tau} \forinfmany (q,r) \in \ZZ^2 }.
$$
A classic application of Dirichlet's pigeonhole principle is that $E(\tau)=[-1,1]$ when $\tau \leq 1$.
%An easy covering argument shows that $\dim_H(E(\tau)) \leq 2/(1+\tau)$ when $\tau > 1$.  
Jarn\'{\i }k \cite{jarnik-1} and Besicovitch \cite{bes} proved that $E(\tau)$ has Hausdorff dimension $2/(1+\tau)$ when $\tau > 1$.
Much further work has been done on metric properties of $E(\tau)$ and various generalizations of it. 
For details, we direct the reader to the recent works  
\cite{allen-troscheit}, 
\cite{beresnevich-bernik-dodson-velani}, 
\cite{beresnevich-dickinson-velani},
and references therein.

Kaufman \cite{kaufman} proved 
%%%\begin{yourthm}
\begin{thm}[Kaufman]
\label{kauf thm}
For every $\tau > 1$, $E(\tau)$ is a Salem set of Hausdorff and Fourier dimension ${2}/({1 + \tau})$. Moreover, there exists a Borel probability measure $\mu$  supported on $E(\tau)$ such that 
$$
|\widehat{\mu}(\xi)| \lesssim |\xi|^{-1/(1+\tau)} \ln(1+|\xi|) \quad \forall \xi \in \RR, \xi \neq 0.
$$
\end{thm}
%%%\end{yourthm}

All known constructions of explicit Salem sets in $\RR^d$ of dimension $\alpha \notin \cbr{0,d-1,d}$ are based on Kaufman's construction.  
Bluhm \cite{bluhm-2} and Hambrook \cite{hambrook-trans} generalized Kaufman's construction to show that some sets closely related to $E(\tau)$ are also Salem sets in $\RR$. 
Bluhm \cite{bluhm-2} also observed that the radial set $\cbr{x \in \RR^d : |x| \in E(\tau)}$ (here and nowhere else $| \; \cdot \;  |$ is the Euclidean norm on $\RR^d$) is a Salem set of dimension $d-1+2/(1+\tau)$ when $\tau > 1$. 
Hambrook \cite{hambrook-R2} generalized Kaufman's construction to give explicit Salem sets in $\RR^2$ of every dimension $0 < \alpha < 2$.

\subsection{Salem Sets and Fourier Dimension: The $p$-adic Setting}

%We now consider the $p$-adic setting. 
Hausdorff dimension in $\QQ_p^d$ is defined exactly as it is in any metric space (see \cite{mattila-book-1}). 
The formula \eqref{haus-form} still holds 
(except $\RR^d$ is replaced by $\QQ_p^d$, and $| \xi |$ is replaced by $| \xi |_{p}$)
%, as 
because 
the proof is based on Frostman's lemma (which holds in any locally compact metric space, see \cite{mattila-book-1}) and properties of the Riesz potential (which still hold in $\QQ_p^d$, see \cite{T75}). Papadimitropoulos \cite{pa-thesis} gives the details in case $d=1$; the proof for $d \geq 2$ is similar. 
The definitions of Fourier dimension and Salem set are as above (with the replacements mentioned).

\pa \cite{pa-p} (see also \cite{pa-thesis}, \cite{pa-local}) adapted Salem's \cite{salem} random Cantor-type construction to prove the existence of Salem sets in $\QQ_p$ of every dimension $0 < \alpha < 1$.

Our first main result, Theorem \ref{thm 0} below, gives explicit Salem sets in $\QQ_p$ of every dimension $0 < \alpha < 1$. It is a $p$-adic version of Theorem \ref{kauf thm}.

For $\tau \in \RR$, the set of $\tau$-well-approximable $p$-adic numbers is 
$$
W(\tau) = \cbr{x \in \ZZ_p : |xq-r|_p \leq \max(|q|,|r|)^{-\tau} \forinfmany (q,r) \in \ZZ^2 }.
$$
The set $W(\tau)$ is a p-adic analogue of $E(\tau)$. Note that the set $E(\tau)$ is unchanged if $\max(|q|,|r|)^{-\tau}$ is replaced by $|q|^{-\tau}$ in the definition. However, if the analogous replacement is made in the definition of $W(\tau)$, the set obtained equals $\ZZ_p$ for all $\tau$. 

%%%To see the last equality, fix any $x \in \ZZ_p$ and $q \in \ZZ$, choose $k \in \NN$ so large that $p^{-k} \leq |q|^{-\tau}$, consider the $p$-adic expansion $qx = \sum_{j=0}^{\infty} c_j p^j$, take $r = \sum_{j=0}^{k-1} c_j p^j$, and (finally) observe that $|qx-r|_p \leq p^{-k} \leq |q|^{-\tau}$.

For $\tau \leq 2$, $W(\tau) = \ZZ_p$ by Dirichlet's pigeonhole principle. 
For $\tau > 2$, Melni{\v{c}}uk \cite{melnicuk} (see also \cite{bovey-dodson}) proved $W(\tau)$ has Hausdorff dimension $2/\tau$.   

Our first main result is 

%%%\begin{ourthm}
\begin{thm}
\label{thm 0}
For every $\tau>2$, $W(\tau)$ is a Salem set of Hausdorff and Fourier dimension ${2}/{\tau}.$ 
Moreover, 
there exists a Borel probability measure $\mu$ supported on $W(\tau)$ such that 
\begin{align*}%%%\label{p kauf a}
|\widehat{\mu}(\xi)| &\lesssim |\xi|_p^{-1/\tau} \ln^2(1+|\xi|_p)  \quad \forall \xi \in \QQ_p, \xi \neq 0.
\end{align*}
\end{thm}
%%%\end{ourthm}

%We have two other results that improve Theorem \ref{thm 0} in different directions.

Our two other main results, Theorem \ref{thm 1} and Theorem \ref{thm 2}, improve Theorem \ref{thm 0} in different ways.

%Our two other main results improve Theorem \ref{thm 0} in different ways.

%
%
%
%

%\subsection{First Improvement: upper regularity and Fourier Restriction}

\subsection{Upper Regularity and Fourier Restriction}

To discuss our first improvement to Theorem \ref{thm 0}, we state a general Stein-Tomas restriction theorem. %%proved by Mitsis \cite{mitsis} and Mockenhaupt \cite{mock}.
\begin{thm}[Mockenhaupt-Mitsis-Bak-Seeger]
%%%\begin{yourthm}
\label{mock-mit thm}
Let $0 < \alpha, \beta < d$. Let $\mu$ be a Borel probability measure on $\RR^d$ such that 
\begin{align}\label{a}
|\widehat{\mu}(\xi)| &\lesssim_{\beta} |\xi|^{-\beta/2}  \quad \forall \xi \in \RR^d, \xi \neq 0, \\
\label{b}
\mu(B(x,r)) &\lesssim_{\alpha} r^{\alpha} \quad \forall x \in \RR^d, r > 0.
\end{align}
%Then 
%%%the Fourier restriction inequality
Then, whenever $1 \leq q \leq 1+\beta/(4d-4\alpha+\beta)$, 
\begin{align}\label{c}
\rbr{\int |\widehat{f}(x)|^{2} d\mu(x)}^{1/2} \lesssim_{\alpha,\beta,q} \| f \|_{q} \quad \forall f \in L^q(\RR^d) \cap L^1(\RR^d).
\end{align}
%whenever $1 \leq q \leq 1+\beta/(4d-4\alpha+\beta)$. 
\end{thm}
%%%\end{yourthm}
%
Note that \eqref{b} is called an upper regularity or Frostman condition, and \eqref{c} is called a Fourier restriction inequality (see \cite{mattila-book-2}, \cite{stein-book} for further background). 

Theorem \ref{mock-mit thm} was proved by Mockenhaupt \cite{mock} and Mitsis \cite{mitsis} for the range $1 \leq q < 1+\beta/(4d-4\alpha+\beta)$. The endpoint was proved by Bak and Seeger \cite{bak-seeger}.

\pa \cite{pa-thesis} extended Kaufman's \cite{kaufman} proof of Theorem \ref{kauf thm} to obtain 
%%%\begin{yourthm} 
\begin{thm}[Papadimitropoulos]
\label{pa thm}
For every $\tau > 1$, $E(\tau)$ is a Salem set with Hausdorff and Fourier dimension ${2}/({1 + \tau})$. Moreover, there exists a Borel probability measure $\mu$ supported on $E(\tau)$ such that 
\begin{align}\label{papadim a}
|\widehat{\mu}(\xi)| &\lesssim |\xi|^{-1/(1+\tau)} \ln(1+|\xi|) \quad \forall \xi \in \RR, \xi \neq 0, 
\\
\label{papadim b}
\mu(B(x,r)) &\lesssim r^{2/(1+\tau)} \ln(1+r^{-1}) \quad \forall x \in \RR, r > 0, 
%\\
%\label{papadim c} \rbr{\int |\widehat{f}(x)|^{2} d\mu(x)}^{1/2} & \lesssim_{q,\tau} \| f \|_{q} \quad \forall f \in L^q(\RR) \cap L^1(\RR), 
\end{align}
%whenever 
%$1 \leq q < 1+\alpha/(4-3\alpha)$, where $\alpha = 2/(1+\tau)$. 
%%%%$1 \leq q \leq 1 + \frac{\frac{2}{1 + \tau}}{4 - \frac{6}{1 + \tau}}$.
%
and, whenever $1 \leq q < 1 + 1/(2(1+\tau) - 3)$,  
\begin{align}
\label{papadim c} \rbr{\int |\widehat{f}(x)|^{2} d\mu(x)}^{1/2} & \lesssim_{q,\tau} \| f \|_{q} \quad \forall f \in L^q(\RR) \cap L^1(\RR).
\end{align}
\end{thm}
%%\end{yourthm}

%Note that Papadimitropoulos actually proved Theorem \ref{pa thm} with slightly weaker versions of \eqref{papadim a} and \eqref{papadim a}, but the proof can be modified to obtain Theorem \ref{pa thm} as stated.

Note that \pa actually proved a version of Theorem \ref{pa thm} with slightly weaker versions of \eqref{papadim a} and \eqref{papadim b}. However, by modifying the proof slightly and using the reduction technique of Section \ref{reduction} below, one may obtain Theorem \ref{pa thm} as stated.

%The majority of the proof of Theorem \ref{pa thm} is proving statement (\ref{papadim b}). 
By Theorem \ref{mock-mit thm}, (\ref{papadim c}) follows from (\ref{papadim a}) and (\ref{papadim b}). 
The main innovation of Theorem \ref{pa thm} over Theorem \ref{kauf thm} is the upper regularity property \eqref{papadim b}.

Theorem \ref{mock-mit thm} also holds in the $p$-adic setting; replace $\mathbb{R}^d$ by $\mathbb{Q}_p^d$ and $|\xi|$ by $|\xi|_p$ in the statement. 
%Indeed, the proof can be translated in a straightforward way from the real to the $p$-adic setting. 
%Basically, the proof is translated from the real to the $p$-adic setting by replacing bump functions with indicator functions in a straightforward way.  
The proof is translated from the real to the $p$-adic setting by replacing bump functions with indicator functions in a straightforward way.
See Papadimitropoulos \cite{pa-thesis} (or \cite{pa-local}) for details in the range $1 \leq q < 1 + {\beta}/({4 d - 4 \alpha + \beta})$. For the endpoint, as in \cite{bak-seeger}, one appeals to the powerful abstract interpolation theorem of Carbery, Seeger, Waigner, and Wright \cite[Section 6.2]{Carbery-Seeger-Waigner-Wright}.

Our second main result (and first improvement to Theorem \ref{thm 0}) is a $p$-adic version of Theorem \ref{pa thm}.

%%%\begin{ourthm}
\begin{thm}
\label{thm 1}
For every $\tau > 2$, $W(\tau)$ is a Salem set with Hausdorff and Fourier dimension ${2}/{\tau}$. Moreover, there exists a Borel probability measure $\mu$ supported on $E(\tau)$ such that 
\begin{align}\label{p kauf a}
|\widehat{\mu}(\xi)| &\lesssim |\xi|_p^{-1/\tau} \ln^2(1 + |\xi|_p) \quad \forall \xi \in \mathbb{Q}_p, \xi \neq 0, 
\\
\label{p kauf b}
\mu(B(x,r)) &\lesssim r^{2/\tau} \ln^2(1+r^{-1}) \quad \forall x \in \mathbb{Q}_p, r > 0,  
%\\
%\label{p kauf c} \rbr{\int |\widehat{f}(x)|^{2} d\mu(x)}^{1/2} & \lesssim_{q,\tau} \| f \|_{q} \quad \forall f \in L^q(\mathbb{Q}_p) \cap L^1(\mathbb{Q}_p), 
\end{align}
%whenever $1 \leq q < 1+\alpha/(4-3\alpha)$, where $\alpha = 2/\tau$.  
%
and, whenever $1 \leq q < 1 + 1/(2\tau - 3)$, 
\begin{align}
\label{p kauf c} \rbr{\int |\widehat{f}(x)|^{2} d\mu(x)}^{1/2} & \lesssim_{q,\tau} \| f \|_{q} \quad \forall f \in L^q(\mathbb{Q}_p) \cap L^1(\mathbb{Q}_p).
\end{align}
%%%\end{ourthm}
\end{thm} 

By the $p$-adic version of Theorem \ref{mock-mit thm}, (\ref{p kauf c}) follows from (\ref{p kauf a}) and (\ref{p kauf b}). 
The main innovation of Theorem \ref{thm 1} over Theorem \ref{thm 0} is the upper regularity property \eqref{p kauf b}. 

Mockenhaupt \cite{mock} (see also \cite{mock-thesis}) proved a version of Theorem \ref{pa thm} for the sets and measures constructed by Salem \cite{salem}. 
Mockenhaupt and Ricker \cite{mock-ricker} then used this theorem to establish an optimal extension of the Hausdorff-Young inequality on the torus $\mathbb{T}$ (which may be identified with $[-1,1]$). 
\pa \cite{pa-p} (see also \cite{pa-thesis}, \cite{pa-local}) proved a 
%%theorem of the same type as 
version of Theorem \ref{thm 1} for the sets and measures given by his $p$-adic analogue of Salem's construction. \pa used that theorem in a manner similar to that of Mockenhaupt and Ricker to establish an optimal extension of the Hausdorff-Young inequality on $\mathbb{Z}_p$.

%
%
%
%
%
%
%

%\subsection{Second Improvement: Higher Dimensions}

\subsection{Multiple Dimensions}

Our second improvement to Theorem \ref{thm 0} generalizes it to multiple dimensions. 

For $m,n \in \NN$, we identify the $m \times n$ matrix whose $ij$-th entry is $x_{ij}$ with the point 
$$
x=(x_{11},\ldots,x_{1n},\ldots,x_{m1},\ldots, x_{mn}). %\in \RR^{mn}.
$$

We first consider a multi-dimensional generalization of $E(\tau)$.  For $\tau \in \RR$, we define 
%$E(m,n,\tau)$ to be the set 
\begin{gather*}
E(m,n,\tau) = \\
\cbr{ x \in [-1,1]^{mn}: \|xq-r\|_p  \leq \max(|q|,|r|)^{-\tau} \forinfmany (q,r) \in \ZZ^n \times \ZZ^m }.
\end{gather*}
%Evidently, $E(1,1,\tau) = E(\tau)$. 
By Minkowski's theorem on linear forms, $E(m,n,\tau) = \RR^{mn}$ when $\tau \leq n/m$. 
Bovey and Dodson \cite{bovey-dodson} showed that the Hausdorff dimension of $E(m,n,\tau)$ is $m(n - 1) + (m + n)/(1 + \tau)$ when $\tau > n/m$. 
The $n=1$ case was done earlier by Jarn{\'\i}k \cite{jarnik-2} and Eggleston \cite{eggleston}. 
 
We mentioned above that Hambrook \cite{hambrook-trans} generalized Kaufman's construction to show that certain sets in $\RR$ closely related to $E(\tau)$ are Salem sets.
In the same paper, 
Hambrook also considered $E(m,n,\tau)$ and proved a version of the following theorem. 
%%%\begin{yourthm}
\begin{thm}[Hambrook]
\label{hambrook thm}
For every $\tau>n/m$, there exists a Borel probability measure $\mu$ supported on $E(m,n,\tau)$ such that 
\begin{align*}%\label{hambrook a}
|\widehat{\mu}(\xi)| &\lesssim |\xi|^{-n/(1+\tau)} \ln^{n}(1+|\xi|) \quad \forall \xi \in \RR^{mn}, \xi \neq 0. \\
\end{align*}
\end{thm}
%%%\end{yourthm}
Technically, 
%this theorem 
Theorem \ref{hambrook thm} 
as stated does not appear in \cite{hambrook-trans}. 
%However, the method of that paper is easily modified to obtain Theorem \ref{hambrook thm}.
However, the proof of Theorem 1.2 of \cite{hambrook-trans} is easily modified to obtain Theorem \ref{hambrook thm}.   
Theorem \ref{hambrook thm} is not strong enough to determine whether $E(m,n,\tau)$ is a Salem set. 
However, it does imply that the Fourier dimension 
%(and hence the Hausdorff dimension) 
of $E(m,n,\tau)$ is at least $2n/(1 + \tau)$. 

We now consider a $p$-adic analogue of $E(m,n,\tau)$ that is a multi-dimensional generalization of $W(\tau)$.  
For $\tau \in \RR$, we define 
\begin{gather*}
W(m,n,\tau) = \\
\cbr{ x \in \ZZ_p^{mn}: \|xq-r\|_p  \leq \max(|q|,|r|)^{-\tau} \forinfmany (q,r) \in \ZZ^n \times \ZZ^m }.
\end{gather*}
Dirichlet's pigeonhole principle 
%(or the $p$-adic version of Minkowski's theorem for linear forms) 
implies $W(m,n,\tau) = \ZZ_p^{mn}$ when $\tau \leq (m+n)/m$. 
Abercrombie \cite{Abercrombie} showed that the Hausdorff dimension of $W(m,n,\tau)$ is $m(n-1) + (m+n)/\tau$ when $\tau > (m+n)/m$.

%Our third and final main result (and second improvement to Theorem \ref{thm 0}) is a multi-dimensional generalization of Theorem \ref{thm 0}. It is a $p$-adic version of Theorem \ref{hambrook thm}. 

Our third 
%and final 
main result (and second improvement to Theorem \ref{thm 0}) is a $p$-adic version of Theorem \ref{hambrook thm}. 

\begin{thm}
%%%\begin{ourthm}
\label{thm 2}
For every $\tau>(m+n)/m$, there exists a Borel probability measure $\mu$ supported on $W(m,n,\tau)$ such that 
\begin{align*}%\label{p kauf a mn}
|\widehat{\mu}(\xi)| &\lesssim |\xi|_p^{-n/\tau} \ln^{n+1}(1+|\xi|_p) \quad \forall \xi \in \QQ_p^{mn}, \xi \neq 0. 
\end{align*}
%%%\end{ourthm}
\end{thm}
Theorem \ref{thm 2} is not strong enough to determine whether $W(m,n,\tau)$ is a Salem set. 
However, it does imply that the Fourier dimension 
%(and hence the Hausdorff dimension) 
of $W(m,n,\tau)$ is at least $2n/\tau$.

By modifying the proof in a straightforward way, it is possible to generalize Theorem \ref{thm 2} even further along the lines of Theorem 1.2 of Hambrook \cite{hambrook-trans}. However, for simplicity, we do not pursue this here.

\subsection{Problems for Future Study}

\begin{prob}\label{prob 1}
For $d \geq 2$, construct Salem sets in $\QQ_p^d$ of every dimension $0 < \alpha < d$. The existence of such sets is unknown. Kahane's \cite{kahane-book} stochastic constructions and Bluhm's \cite{bluhm-1} Cantor-type construction of Salem sets in $\RR^d$ are good candidates for adaptation to the $p$-adic setting.
\end{prob}

\begin{prob}
Determine the Fourier dimension of $W(m,n,\tau)$ when $\tau > (m+n)/n$ and $mn > 1$. As mentioned above, the Hausdorff dimension of $W(m,n,\tau)$ is known to be $m(n-1)+(m+n)/\tau$, and Theorem \ref{thm 2} implies the Fourier dimension of $W(m,n,\tau)$ is at least $2n/\tau$. By improving on the method of the present paper, perhaps it is possible to show that $\dim_F W(m,n,\tau) \geq m(n-1)+(m+n)/\tau$, hence proving that $W(m,n,\tau)$ is Salem. Note that this would also resolve Problem \ref{prob 1}. On the other hand, it would be interesting to obtain an upper bound on $\dim_F W(m,n,\tau)$ that is strictly less than the Hausdorff dimension, as such upper bounds appear to be difficult. 
The analogous problem for $E(m,n,\tau)$ is also open.
\end{prob}

\begin{prob}
Prove an analogue of Theorem \ref{thm 1} for $W(m,n,\tau)$. 
In other words, prove Theorem \ref{thm 2} with an analogue of the upper regularity property \eqref{p kauf b} (an analogue of \eqref{p kauf c} would follow immediately from the $p$-adic version of Theorem \ref{mock-mit thm}).  
The analogue of the upper regularity property \eqref{p kauf b} 
would take the form 
%would be an inequality of the form
$$
\mu(B(x,r)) \lesssim r^{\alpha} \quad \forall x \in \QQ_{p}^{mn}, r > 0.
$$
In the case $m > n = 1$, the best possible exponent $\alpha$ is $\alpha = (m+1)/\tau$. 
The method of proof of Theorem \ref{thm 1} can be extended to obtain 
this, 
but we must assume $\tau > (m+1)/m + 1 - 1/m^2$. 
In full range $\tau > (m+1)/m$, we are only able to obtain $\alpha = m/\tau$. 
The case $n > 1$ is completely open. 
%The analogous problem for $E(m,n,\tau)$ is also open. 
The analogous problem for $E(m,n,\tau)$ %may also be considered. 
The analogous problem for $E(m,n,\tau)$ %may also be considered. 
is also interesting to consider. 
\end{prob}

\begin{prob}
Prove versions of Theorem \ref{thm 0}, Theorem \ref{thm 1}, and Theorem \eqref{thm 2} in the setting of an arbitrary ultrametric local field. Note that every local field is isomorphic to either $\RR$, $\CC$, $\QQ_p$ (for some prime $p$), a finite extension of $\QQ_p$ (for some prime $p$), or the field of formal Laurent series over some finite field, and $\RR$ and $\CC$ are not ultrametric. 
\pa \cite{pa-thesis}, \cite{pa-local} extended Salem's \cite{salem} random Cantor-type construction to prove, for any ultrametric local field $K$, the existence of Salem sets of every dimension $0 < \alpha < 1$ in $K$. Moreover, \pa \cite{pa-thesis}, \cite{pa-local} proved a version of Theorem \ref{thm 1} in $K$ for the sets and measures produced by his construction. 
\end{prob}

\subsection{Structure of the Paper}

In Section \ref{p-adic}, we review the definition and basics properties of the $p$-adic numbers as well as the necessary elements of Fourier analysis on the $p$-adics. 
In Sections \ref{proof thm 1} and \ref{proof thm 2}, we prove Theorems \ref{thm 1} and \ref{thm 2}, respectively. 
Theorem \ref{thm 0} is an immediate corollary of both Theorem \ref{thm 1} and \ref{thm 2}. 

\subsection{Remarks on the Proofs}

The proof of Theorem \ref{thm 1} is a reasonably straightforward adaptation of \pa's \cite{pa-thesis} proof of Theorem \ref{pa thm}, which in turn is an extension of Kaufman's \cite{kaufman} proof of Theorem \ref{kauf thm}, from the real to the $p$-adic setting.  
%Basically, the adaptation strategy boils down to replacing a bump function that is $1$ on $[-1,1]=B(0,1) \subseteq \RR$ by the indicator function of $\ZZ_p=B(0,1) \subseteq \QQ_p$.
%%%%%%%%The basic idea is to replace a bump function that is $1$ on $[-1,1]=B(0,1) \subseteq \RR$ by the indicator function of $\ZZ_p=B(0,1) \subseteq \QQ_p$. 
%As usual, the devil is in the details. 
%In this case, the devil is an exponential sum which appears in the proof of Lemma \ref{FM lemma} below. 
%%%%%%%%We estimate it by a method inspired by Theorem 1 in Cilleruelo and Garaev's paper \cite{cill-gar}. 
%No such obstacle is encountered in the real setting. 
%We estimate the exponential sum by a method inspired by Theorem 1 in Cilleruelo and Garaev's paper \cite{cill-gar}. 
%
In essence, the adaptation strategy is to replace a bump function that is $1$ on $[-1,1]=B(0,1) \subseteq \RR$ by the indicator function of $\ZZ_p=B(0,1) \subseteq \QQ_p$. The details, however, are not completely straightforward. 
In establishing \eqref{p kauf a}, we encounter (in the proof of Lemma \ref{FM lemma} below) a non-trivial exponential sum. We estimate the exponential sum by a method inspired by Theorem 1 in Cilleruelo and Garaev's paper \cite{cill-gar}. 
No such obstacle is encountered in the real setting. 
%Establishing \eqref{p kauf b} is also somewhat different than in the real setting, as we use the unique arithmetic and geometric properties of the $p$-adic numbers. 
Establishing \eqref{p kauf b} is also somewhat different than in the real setting because of the unusual geometry of the $p$-adic numbers. 
%%% in important ways. 

Note that the reduction technique of Section \ref{reduction} below, while simple,  appears to be new. It allows us to obtain the strong Fourier decay and upper regularity inequalities \eqref{p kauf a} and \eqref{p kauf b} without the averaging technique of Kaufman \cite{kaufman}. Using Kaufman's averaging technique would make proving \eqref{p kauf b}, even in a weaker form, significantly more complicated. \pa \cite{pa-thesis} did not use Kaufman's averaging argument to prove his version of Theorem \ref{pa thm}, which (as we mentioned above) has weaker forms of \eqref{p kauf a} and \eqref{p kauf b}.

The proof of Theorem \ref{thm 2} is a generalization of the proof of Theorem \ref{thm 1} (without the upper regularity property \eqref{p kauf b}), following the ideas of \cite{hambrook-trans}.

\section{The field $\QQ_p$ of $p$-adic Numbers}\label{p-adic}

\subsection{Definition and Basic Properties}\label{p-adic basics}

Every non-zero $x \in \QQ$ can be expressed uniquely in the form $x=p^{M}a/b$ where $a,b,M$ are integers with $a$ and $b$ coprime to $p$ and $b \geq 1$. 
The $p$-adic absolute value of $x$ is defined to be $|x|_p = p^{-M}$. We define $|0|_p = 0$. The completion of $\QQ$ with respect to the $p$-adic absolute value is the field of $p$-adic numbers $\QQ_p$. 
Every non-zero 
$x \in \QQ_p$   
can be expressed uniquely in the form 
\begin{align}\label{p-adic expansion}
x = \sum_{j=M}^{\infty} c_j p^j,
\end{align}
where $M \in \ZZ$, $c_j \in \cbr{0,1,\ldots,p-1}$, and $c_M \neq 0$. 
We call \eq{p-adic expansion} the $p$-adic expansion of $x$. 
The $p$-adic absolute value of $x$ is $|x|_p = p^{-M}$. 
This extends the definition of the $p$-adic absolute from $\QQ$ to $\QQ_p$. 
It is sometimes helpful to know that $|x|_p \geq |x|^{-1}$ for all non-zero $x \in \QQ$.

%Every non-zero 
%$x \in \QQ_p$   
%can be expressed uniquely in the form 
%\begin{align}\label{p-adic expansion}
%x = \sum_{j=M}^{\infty} c_j p^j,
%\end{align}
%where $M \in \ZZ$, $c_j \in \cbr{0,1,\ldots,p-1}$, and $c_M \neq 0$. We call \eq{p-adic expansion} the $p$-adic expansion of $x$. The $p$-adic absolute value of $x$ is defined %to be $|x|_p = p^{-M}$. 
%We define $|0|_p = 0$. 
%Note that $\QQ_p$ is the completion of $\QQ$ with respect to $| \cdot |_p$.
%It is sometimes helpful to know that $|x|_p \geq |x|^{-1}$ for all non-zero $x \in \QQ$.  

%The $p$-adic norm of $x \in \QQ_p^d$ is $|x|_p = \max_{1 \leq i \leq d} |x_i|_p$.
%The closed ball with radius $r > 0$ and center $a \in \QQ_p^d$ is $B(a,r) = \cbr{x \in \QQ_p^d : |x - a|_p \leq r}$. 
%Since the $p$-adic norm takes values in a discrete set 
%%%%%%actually 0 is limit point
%$B(a,r)$ is both a closed set and an open set, and 
%the indicator function $\one_{B(a,r)}$ is continuous. 

The $p$-adic norm of $x \in \QQ_p^d$ is $|x|_p = \max_{1 \leq i \leq d} |x_i|_p$.

The closed ball with radius $r > 0$ and center $a \in \QQ_p^d$ is $B(a,r) = \cbr{x \in \QQ_p^d : |x - a|_p \leq r}$. 
%The closed (resp. open) ball with radius $r > 0$ and center $a \in \QQ_p^d$ is $B(a,r) = \cbr{x \in \QQ_p^d : |x - a|_p \leq r}$ (resp. $B(a,r^{-}) = \cbr{x \in \QQ_p^d : |x - a|_p < r}$). 
Since the $p$-adic norm takes values in $\cbr{p^k: k \in \ZZ} \cup \cbr{0}$, it follows that 
$B(a,r) = \cbr{x \in \QQ_p^d : |x-a|_p \leq p^{k}} = \cbr{x \in \QQ_p^d : |x-a|_p < p^{k+1}}$ whenever $p^k < r \leq p^{k+1}$, 
%$B(a,r)$ is both a closed set and an open set, 
and the indicator function $\one_{B(a,r)}$ is continuous. 
Analogous statements hold for the open ball $B(a,r^{-}) = \cbr{x \in \QQ_p^d : |x - a|_p < r}$. 
%Analogous statements are true about the open ball $B(a,r^{-})$.
%
%
%For any $a \in \QQ_p^d$, $k \in \ZZ$, and $p^k < r < p^{k+1}$, $B(a,r) = B(a,r^{-}) = B(a,p^k)$. If $r=p^{k+1}$, B(a,r^{-})
%For any $a \in \QQ_p^d$, $r > 0$, and $k \in \ZZ$, we have $B(a,r) = B(a,p^k)$ whenever $p^k \leq r < p^{k+1}$, and $B(a,r^{-}) = B(a,p^k)$ whenever $p^k < r \leq p^{k+1}$. 
%$B(a,r) = B(a,p^k)$ whenever $p^k \leq r < p^{k+1}$, and $B(a,r^{-}) = B(a,p^k)$ whenever $p^k < r \leq p^{k+1}$.

The $p$-adic norm satisfies a strong form of the triangle inequality:  
\begin{align*}
|x-y|_p 
%\leq \max(|x|_p,|y|_p),  \text{ with equality whenever } |x|_p \neq |y|_p. 
\leq \max(|x|_p,|y|_p) \qquad \forall x,y\in \QQ_p, \text{ with equality whenever } |x|_p \neq |y|_p.
\end{align*} 
This inequality is called  
the ultrametric inequality. 
It may also be called 
the acute isosceles triangle inequality 
because it means precisely that for each $x,y \in \QQ_p^d$ the two largest of $|x|_p$, $|y|_p$, $|x-y|_p$ are equal. 
%because it is equivalent to the statement that 
%for each $x,y,z \in \QQ_p^d$ the two largest of $|x-y|_p$, $|y-z|_p$, $|z-x|_p$ are equal.

%The ultrametric inequality implies an important property of balls in $\QQ_p^d$: For all $a,a' \in \QQ_p^d$ and all $0 < r \leq r'$, $B(a,r) \cap B(a',r') \neq \emptyset$ if and only if $B(a,r) \subseteq B(a',r')$. In words, two balls intersect if and only if the larger contains the smaller. 

The ultrametric inequality implies two important properties of balls in $\QQ_p^d$. 
%
%First, two balls intersect if and only if the larger contains the smaller. More precisely, for all $a,a' \in \QQ_p^d$ and all $0 < r \leq r'$, $B(a,r) \cap B(a',r') \neq \emptyset$ if and only if $B(a,r) \subseteq B(a',r')$. In wq
%
First, for all $a,a' \in \QQ_p^d$ and all $0 < r \leq r'$, $B(a,r) \cap B(a',r') \neq \emptyset$ if and only if $B(a,r) \subseteq B(a',r')$. In words, two balls intersect if and only if the larger contains the smaller.  
The second property is that, for all integers $j < k$, every ball in $\QQ_p^d$ of radius $p^k$ is the union of $p^{d(k-j)}$ balls of radius $p^j$. Indeed, for every $x \in \QQ_p^d$, we have $B(x,p^k) = \bigcup_m B(x + y p^{-k}, p^j)$, where the union runs over all $y \in \QQ_p^d$ such that $y_i \in \cbr{0,1,\ldots,p^{k-j}-1}$. 
From these properties, it follows that every ball in $\QQ_p^d$ is compact; hence, $\QQ_p^d$ is locally compact.

The closed unit ball in $\QQ_p$, $B(0,1) = \cbr{x \in \QQ_p : |x|_p \leq 1}$, is called the ring of $p$-adic integers and is denoted $\ZZ_p$. 
Thus $\ZZ_p^d = B(0,1) = \cbr{x \in \QQ_p^d : |x|_p \leq 1}$. 

For nonzero $x \in \QQ_p$ with $p$-adic expansion \eq{p-adic expansion}, the $p$-adic fractional part of $x$ is defined to be $\cbr{x}_p = \sum_{j=M}^{-1} c_j p^j$. 
We define $\cbr{0}_p = 0$. 
For nonzero $x \in \QQ_p$ with $p$-adic expansion \eq{p-adic expansion}, the $p$-adic integral part of $x$ is defined to be $[x]_p = \sum_{j=0}^{\infty} c_j p^j$. 
We define $\sbr{0}_p = 0$. 
Notice $x = \cbr{x}_p + [x]_p$ for all $x \in \QQ_p$. Moreover, $x \in \ZZ_p$ if and only if $\cbr{x}_p = 0$, which is the case if and only if $[x]_p=x$. 
For all $x,y \in \QQ_p$, $\cbr{x}_p + \cbr{y}_p$ differs from $\cbr{x + y}_p$ by an integer, and so 
\begin{align}\label{additive}
e(\cbr{x}_p + \cbr{y}_p) = e(\cbr{x + y}_p).
\end{align}
We identify $\QQ_p / \ZZ_p$ with the set $\cbr{x \in \QQ_p : [x]_p = 0} \subseteq \QQ \cap [0,1)$.

\subsection{Fourier Analysis on $\QQ_p^d$}\label{p-adic Fourier}

We review here the necessary elements of Fourier analysis on $\QQ_p^d$. 
The books by Folland \cite{folland-book-abstract} and Taibleson \cite{T75} are excellent general references on the subject.

The additive group $(\QQ_p^d,+)$ is a commutative locally compact Hausdorff topological group. 
We denote by $dx$ the unique Haar measure on $\QQ_p^d$ that assigns measure $p^{dk}$ to every closed ball of radius $p^{k}$, $k \in \ZZ$. The Haar measure satisfies the following scaling property: $d(ax)=|a|_p^d dx$ for all $a \in \QQ_p$. 
The Haar measure on $\QQ_p^d$ is the $d$-fold product of the corresponding Haar measure on $\QQ_p$, which we also denote by $dx$. 

The characters on a commutative locally compact Hausdorff topological group are the continuous homomorphisms from the group to the unit circle in $\CC$ (which is a group under multiplication). 
By \eqref{additive}, 
$
x \mapsto e(\cbr{x \cdot s}_p)
$
is a character for every $s \in \QQ_p^d$. 
In fact, every character on $\QQ_p^d$ is of this form. 
%Every character on $\QQ_p^d$ has the form
%$
%x \mapsto e(\cbr{x \cdot s}_p)
%$   
%for some $s \in \QQ_p^d$. 
%%%%Thus, if 
If $f: \QQ_p^d \rightarrow \CC$ is integrable, the Fourier transform of $f$ is 
$$
\widehat{f}(s) = \int_{\QQ_p^d} e(\cbr{x \cdot s}_p) f(x) dx \quad \forall s \in \QQ_p^d.
$$
If $\mu$ is a finite Borel measure on $\QQ_p^d$, the Fourier transform of $\mu$ is
$$
\widehat{\mu}(s) = \int_{\QQ_p^d} e(\cbr{x \cdot s}_p) d\mu(x) \quad \forall s \in \QQ_p^d.
$$
%

%If $\mu$ is a measure on $\QQ_p^d$ and $f: \QQ_p^d \rightarrow \CC$ is integrable with respect to $\mu$, the Fourier transform of $fd\mu$ is 
%$$
%\widehat{fd\mu}(s) = \int_{\QQ_p^d} e(\cbr{x \cdot s}_p) f(x) d\mu(x) \quad \forall s \in \QQ_p^d.
%$$
%$\widehat{f} = \widehat{fdx}$ and $\widehat{\mu} = \widehat{1d\mu}$. 

The Haar measure on $\ZZ_p^d$ is the restriction of the Haar measure on $\QQ_p^d$. 
Every character on $\ZZ_p^d$ has the form  
$
x \mapsto e(\cbr{x \cdot s}_p)
$   
for some $s \in (\QQ_p / \ZZ_p)^d$. 
%Thus, if 
If 
$f: \ZZ_p^d \rightarrow \CC$ is integrable, the Fourier transform of $f$ is 
$$
\widehat{f}(s) = \int_{\ZZ_p^d} e(\cbr{x \cdot s}_p) f(x) dx \quad \forall s \in (\QQ_p / \ZZ_p)^d.
$$
If $\mu$ is a finite Borel measure on $\ZZ_p^d$, the Fourier transform of $\mu$ is 
$$
\widehat{\mu}(s) = \int_{\ZZ_p^d} e(\cbr{x \cdot s}_p) d\mu(x) \quad \forall s \in (\QQ_p / \ZZ_p)^d.
$$

%We now present as lemmas several properties of the Fourier transform that we will need. 

We now present two lemmas that we will need. The first is a simple calculation.

\begin{lem}\label{ball lemma}
For every $k \in \ZZ$, $a \in \QQ_p^d$, and $s \in \QQ_p^d$, we have
\begin{align}\label{ball lemma 1}
\int_{B(a,p^{-k})} e(\cbr{s \cdot x}_p) dx
= 
\left\{
\begin{array}{cl}
p^{-dk} e(\cbr{s \cdot a}_p) & \text{if } |s|_p \leq p^k \\
0 & \text{if } |s|_p > p^k. 
\end{array} \right.
\end{align}
\end{lem}
\begin{proof}

By a change of variable, 
$$\int_{B(a,p^{-k})} e(\cbr{s \cdot x}_p) dx = p^{-dk} e(\cbr{s \cdot a}_p) \int_{B(0,1)} e(\cbr{p^k s \cdot x}_p) dx,$$ 
so it will suffice to prove \eqref{ball lemma 1} when 
$a=0$ and $k=0$. 
As the $d > 1$ case follows from the $d=1$ case, we will also assume $d=1$. 
If $|s|_p \leq 1$, then $\cbr{s x}_p = 0$ for all $x \in B(0,1)$, and so 
$\int_{B(0,1)} e(\cbr{sx}_p) dx = 1.$ 
Now suppose $|s|_p > 1$. 
By first making a change of variable and then using that $B(-1,1)=B(0,1)$, we get 
\begin{align*}
\int_{B(0,1)} e(\cbr{sx}_p) dx = e(\cbr{s}_p) \int_{B(-1,1)} e(\cbr{sx}_p) dx = e(\cbr{s}_p) \int_{B(0,1)} e(\cbr{sx}_p) dx.
\end{align*}
Therefore, since $e(\cbr{s}_p) \neq 1$, we must have $\int_{B(0,1)} e(\cbr{sx}_p) dx = 0$.
\end{proof}
%
%
%
%
%We will need the following special case of the Fourier inversion formula 
%(see \cite[p.102]{folland-book-abstract}, \cite[p.120]{T75}).  
%\begin{lemma}\label{inversion}
%If $f:\ZZ_p^d \to \CC$ is continuous and $\sum_{s \in (\QQ_p/\ZZ_p)^d} |\widehat{f}(s)|$ is finite, then 
%then 
%\begin{align*}
%f(x) = \sum_{s \in (\QQ_p/\ZZ_p)^d} e(-\cbr{x \cdot s}_p) \widehat{f}(s) \quad \forall x \in \ZZ_p^d. 
%\end{align*}
%\end{lemma}

%The final 
The second lemma is the $p$-adic version of a lemma of Kahane (see \cite[pp.252-253]{kahane-book}) whose proof is easily translated from $\RR^d$ to $\QQ_p^d$.
\begin{lem}\label{Kahane-transfer}
Let $\mu$ be a Borel measure on $\QQ_p^d$ with support contained in $\ZZ_p^d$, and let $\phi$ and $\psi$ be positive non-increasing functions defined on $(0,\infty)$ such that $\phi(t/2) \lesssim \phi(t)$ and $\psi(t/2) \lesssim \psi(t)$ for all $t > 0$. If $|\widehat{\mu}(s)| \lesssim \phi(|s|_p)/\psi(|s|_p)$ for all $s \in (\QQ_p/\ZZ_p)^d$, then $|\widehat{\mu}(s)| \lesssim \phi(|s|_p)/\psi(|s|_p)$ for all $s \in \QQ_p^d$.
\end{lem}

\section{Proof of Theorem \ref{thm 1}}\label{proof thm 1}

\subsection{Reduction}\label{reduction}

%Consider the following seemingly weaker version of Theorem \ref{thm 1}.
We show here that to prove Theorem \ref{thm 1} it suffices to prove the seemingly weaker. 

\begin{thm}\label{thm 1a}
Let $g$ be a non-negative non-decreasing function defined on $(0,\infty)$ such that $\lim_{x \rightarrow \infty} g(x)=\infty$. 
For every $\tau>2$, there exists a Borel probability measure $\mu$ supported on $W(\tau)$ such that 
\begin{align}\label{p kauf a g}
|\widehat{\mu}(\xi)| &\lesssim |\xi|_p^{-1/\tau} \ln^2(1+|\xi|_p) g(|\xi|_p)  \quad \forall \xi \in \QQ_p, \xi \neq 0. \\
\label{p kauf b g}
\mu(B(x,r)) &\lesssim r^{2/\tau} \ln(1+r^{-1}) g(r^{-1}) \quad \forall x \in \QQ_p, r>0.
\end{align}
%where the constant implied by $\lesssim$ depends on $\tau$ but not $g$.
We emphasize that the constant implied by $\lesssim$ does not depend on $g$.
\end{thm}

%%%%%We prove Theorem \ref{thm 1a} in Section \ref{proof of 1a}. Here we show that Theorem \ref{thm 1a} implies Theorem \ref{thm 1}.

%We show here that Theorem \ref{thm 1a} implies Theorem \ref{thm 1}. We prove Theorem \ref{thm 1a} in Section \ref{proof of 1a}. 

We prove Theorem \ref{thm 1a} in Section \ref{proof of 1a}. 

\begin{proof}[Proof that Theorem \ref{thm 1a} implies Theorem \ref{thm 1}]
Let $\tau > 2$. 
%For $t > 0$ and $n \in \NN$, define $g_1(t) = \ln(1+t)$ and $g_n(t) = \ln(1+g_{n-1}(t))$.
For each $k \in \NN$, Theorem \ref{thm 1a} gives a probability measure $\mu_k$ supported on $W(\tau)$ 
that satisfies 
%such that 
%\begin{align*}
%%%\label{p kauf a g}
%|\widehat{\mu_n}(\xi)| &\lesssim |\xi|_p^{-1/\tau} \ln^2(1+|\xi|_p) g_n(1+|\xi|_p)  \quad \forall \xi \in \QQ_p, \xi \neq 0 \\
%%%\label{p kauf b}
%\mu_n(B(x,r)) &\lesssim r^{2/\tau} \ln(1+r^{-1}) g_n(1+r^{-1}) \quad \forall x \in \QQ_p, r>0.
%\end{align*}
\eqref{p kauf a g} and \eqref{p kauf b g} with $\mu$ and $g(t)$ replaced by $\mu_k$ and $\ln^{1/k}(1 + t)$, respectively.  
By Prohorov's theorem (see \cite[vol.2, p.202]{bogachev}), the sequence $(\mu_k)_{k=1}^{\infty}$ has a subsequence $(\mu_{k_j})_{j=1}^{\infty}$ 
which converges weakly (that is, in distribution) to a probability measure $\mu$. 
%Redefine $(\mu_n)_{n=1}^{\infty}$ to be that subsequence. 
%Weak convergence implies the measure $\mu$ satisfies 
Therefore 
$\widehat{\mu}(\xi) = \lim_{j \rightarrow \infty} \widehat{\mu_{k_j}}(\xi)$ for all $\xi \in \QQ_p$, 
and $\mu(B(x,r)) = \lim_{j \rightarrow \infty}\mu_{k_j}(B(x,r))$ for all $x \in \QQ_p$, $r>0$ (because $B(x,r)$ is both open and closed). 
It follows that $\mu$ satisfies \eqref{p kauf a} and \eqref{p kauf b} 
%because for any fixed $t > 0$, we have $g_n(t) \leq 2$ for all sufficiently large $n \in \NN$.
because $\lim_{k \rightarrow \infty} \ln^{1/k}(1+t) = 1$ for any fixed $t > 0$. 
This proves Theorem \ref{thm 1}.
\end{proof}

\subsection{Proof of Theorem \ref{thm 1a}}\label{proof of 1a}

Let $\tau > 2$. Let $g$ be any non-negative non-decreasing function defined on $(0,\infty)$ such that $\lim_{x \rightarrow \infty} g(x)=\infty$.   
For each $M \in \NN$, define 
\begin{align*}
Q_M &= \cbr{q \in \ZZ : \frac{1}{2}p^{M} \leq q < p^M, \; |q|_p = 1, \; q \; \text{prime}}, \\
R_M &= \cbr{r \in \ZZ : 0 \leq r < p^M}.
\end{align*}
Note that $Q_M$ is non-empty unless $p=2$ and $M=1$. For everything that follows, we make the standing assumption that $M \geq 2$ if $p=2$. 
%Therefore we make the following standing assumption for all definitions and lemmas below: If $p=2$, we must have $M \geq 2$. 
For each $q \in Q_M$ and $r \in R_M$, define the function $\phi_{q,r}$ on $\ZZ_p$ by 
$$
%\phi_{q,r}(x) = p^{\ceil{M\tau}} \one_{B(0,1)}(p^{-\ceil{\tau M}}(xq-r)) = p^{\ceil{M\tau}} \one_{B(r/q,p^{-\ceil{\tau M}})}(x) \quad \forall x \in \ZZ_p.
\phi_{q,r}(x) = p^{\ceil{\tau M}} \one_{B(0,1)}(p^{-\ceil{\tau M}}(xq-r)) \quad \forall x \in \ZZ_p.
$$
For each $M \in \NN$, define the function $F_M$ on $\ZZ_p$ by 
$$
F_M(x) 
= |Q_M|^{-1}|R_M|^{-1} \sum_{q \in Q_M} \sum_{r \in R_M} \phi_{q,r}(x)
\quad \forall x \in \ZZ_p.
$$
Choose a strictly increasing sequence of non-negative integers $(M_k)_{k=0}^{\infty}$ such that for all $k \in \NN$
\begin{align}
\label{Mk size 1}
\ceil{\tau M_{k-1}} &< M_k, \\
\label{Mk size 2}
p^{\ceil{\tau M_{k-1}}} &< g(p^{M_k}), \\
\label{Mk size 3}
\prod_{i=1}^{k-1} \frac{p^{\ceil{\tau M_i}}}{|Q_{M_i}||R_{M_i}|} &< g(p^{M_k}).
\end{align}
Let $\psi_0$ be any non-negative function on $\ZZ_p$ such that
\begin{align}
\label{psi0 normal}
\widehat{\psi_0}(0) &= 1, \\
\label{psi0 decay}
\widehat{\psi_0}(s) &= 0 \quad \text{ for all } s \in \QQ_p/\ZZ_p \text{ with }|s|_p > p^{\ceil{\tau M_0}}, \\
\label{psi0 support} 
\psi_0(x) &= 0 \quad \text{ for all } x \in \ZZ_p \text{ with } |x|_p \leq p^{-\ceil {\tau M_1}} \text{ or } |x|_p=1, \\
\label{psi0 bound} 
\| \psi_0 \|_{\infty} &< \infty.
\end{align}
In light of Lemma \ref{ball lemma}, we may choose, for example, 
$$
\psi_0 = (p^{-1} - p^{-2})^{-1} (\one_{B(0,p^{-1})} - \one_{B(0,p^{-2})}).
$$
For each $k \in \NN$, define the measure $\mu_k$ on $\ZZ_p$ by  
$$
d \mu_k(x) = \psi_0(x) F_{M_1}(x) \cdots F_{M_k}(x)  dx.
$$ 
For convenience in Lemma \ref{mu-k lemma} below, 
we define $d\mu_{-1}(x) = d\mu_0(x) = \psi_0(x) dx$.

%The proof proceeds by the following sequence of lemmas.

To construct the measure $\mu$ and prove that it satisfies \eqref{p kauf a g}, we need the following sequence of lemmas.

\begin{lem}\label{phi lemma}
For all $M \in \NN$, $q \in Q_M$, $r \in R_M$, and $s \in \QQ_p / \ZZ_p$, 
$$
\widehat{\phi_{q,r}}(s) = \left\{
\begin{array}{cl}
e( \cbr{ rs / q }_p  ) & \text{if } |s|_p \leq p^{\ceil{\tau M}}\\
0 & \text{if } |s|_p > p^{\ceil{\tau M}}
\end{array} \right.
$$
\end{lem}

\begin{lem}\label{FM lemma}
For all $M \in \NN$ and $s \in \QQ_p / \ZZ_p$, 
\begin{IEEEeqnarray}{rCll}
\label{FM1}
\widehat{F_M}(s) & = &  1  & \quad \text{if }  s = 0  \\
\label{FM2}
\widehat{F_M}(s) & = & 0  & \quad \text{if }   0  <  |s|_p  \leq  p^M \\
\label{FM3}
|\widehat{F_M}(s)| &\lesssim&  |s|^{-1/\tau} \ln^2 (|s|_p)  & \quad \text{if }   p^M  <  |s|_p  \leq  p^{\ceil{\tau M}} \\
\label{FM4}
\widehat{F_M}(s) & = & 0 \quad  & \quad \text{if }  |s|_p  >  p^{\ceil{\tau M}} 
\end{IEEEeqnarray}
\end{lem}

\begin{lem}\label{mu-k lemma}
For all integers $k \geq 0$ and all $s \in \QQ_p / \ZZ_p$, 
\begin{IEEEeqnarray}{rCll}
\label{mu-k 1}
\widehat{\mu_k}(s) & = & 1 & \quad \text{if } s = 0 \\
\label{mu-k 2}
\widehat{\mu_k}(s) & = & \widehat{\mu_{k-1}}(s) & \quad \text{if } 0  <  |s|_p  \leq  p^{M_k} \\
\label{mu-k 3}
|\widehat{\mu_k}(s)| 
&\lesssim & |s|^{-1/\tau} \ln^2 (|s|_p) g(|s|_p)
& \quad \text{if } p^{M_k}  <  |s|_p  \leq  p^{\ceil{\tau M_k}} \\
\label{mu-k 4}
\widehat{\mu_k}(s) & = & 0 & \quad \text{if }  |s|_p  >  p^{\ceil{\tau M_k}}
\end{IEEEeqnarray}
\end{lem}

Lemma \ref{phi lemma} is an immediate corollary of Lemma \ref{ball lemma}. 
The proofs of Lemmas \ref{FM lemma} and \ref{mu-k lemma} are given in 
%Section \ref{proofs of lemmas}. 
Sections \ref{proofs of lemmas 1} and \ref{proofs of lemmas 2}, respectively.

Note that \eqref{mu-k 1} implies that each $\mu_k$ is a probability measure. 
By Prohorov's theorem (see \cite[vol.2, p.202]{bogachev}), the sequence $(\mu_k)_{k=1}^{\infty}$ has a subsequence 
that converges weakly (that is, in distribution) to a probability measure $\mu$. 
Though $\mu$ is technically a measure on $\ZZ_p$, it extends to a measure on $\QQ_p$ by defining $\mu(A)=\mu(A \cap \ZZ_p)$ for $A \subseteq \QQ_p$. 

%We have
%\begin{gather*}
%\text{supp}(\mu) 
%\subseteq \bigcap_{k=1}^{\infty} \text{supp}(F_{M_k}) \\
%= \bigcap_{k=1}^{\infty} \cbr{x \in \ZZ_p : |xq-r|_p \leq p^{-\ceil{\tau M_{k}}} \text{ for some } (q,r) \in Q_{M_k} \times R_{M_k}} 
%\subseteq W(\tau),
%\end{gather*}
%where the last containment holds because \eqref{Mk size 1} implies that $Q_{M_k} \times R_{M_k}$ and $Q_{M_{k'}} \times R_{M_{k'}}$ are disjoint 
%%%for all positive integers $k < k'$. 
%for any two $k,k' \in \NN$.

Since 
$$
\text{supp}(F_{M_k}) 
= \cbr{x \in \ZZ_p : |xq-r|_p \leq p^{-\ceil{\tau M_{k}}} \text{ for some } (q,r) \in Q_{M_k} \times R_{M_k}} 
$$
for any $k \in \NN$, and since \eqref{Mk size 1} implies that $Q_{M_k} \times R_{M_k}$ and $Q_{M_{k'}} \times R_{M_{k'}}$ are disjoint 
%for all positive integers $k < k'$. 
for any two $k,k' \in \NN$, 
we have  
$$
\text{supp}(\mu) \subseteq \bigcap_{k=1}^{\infty} \text{supp}(F_{M_k}) \subseteq W(\tau).
$$ 
%
%

%Furthermore, by 
By 
\eqref{psi0 decay} and 
\eqref{mu-k 2}-\eqref{mu-k 4},
$$
|\widehat{\mu}(s)| 
\leq \sup_{k \in \NN} |\widehat{\mu_k}(s)| 
\lesssim |s|_p^{-1/\tau} \ln^2 (|s|_p) g(|s|_p)
\quad \forall s \in \QQ_p/\ZZ_p, s \neq 0.
$$
%
%Though $\mu$ is technically a measure on $\ZZ_p$, it extends to a measure on $\QQ_p$ by defining $\mu(A)=\mu(A \cap \ZZ_p)$ for $A \subseteq \QQ_p$. 
%
An application of Lemma \ref{Kahane-transfer} shows that $\mu$ satisfies $\eqref{p kauf a g}$.

Now we move on to proving $\eqref{p kauf b g}$. 

Since $\mu$ is a probability measure supported on $\ZZ_p$, and since every closed ball in $\ZZ_p$ can be written in the form $B(x,p^{-\ell})$ with $x \in \ZZ_p$ and 
%$\ell \geq 0$, 
%$\ell$ a non-negative integer, 
$\ell \in \ZZ$, $\ell \geq 0$,  
it suffices to prove 
\begin{align*}%\label{reg thm 1}
\mu(B(x,p^{-\ell})) \lesssim p^{-2\ell/\tau}\ln(1 + p^{\ell}) g(p^{\ell}) \quad \forall x \in \ZZ_p, \ell \in \ZZ, \ell \geq 0.
\end{align*}
We can reduce things further. 
%Let $x \in \ZZ_p$ and $\ell \in \NN$. 
%If $\ell \leq \ceil{\tau M_0}$, then 
If $x \in \ZZ_p$ and $0 \leq \ell \leq \ceil{\tau M_0}$, then 
$$
%\mu(B(x,p^{-\ell})) \leq 1 \leq p^{2\ceil{\tau M_0}/\tau} p^{-2\ell/\tau}, 
\mu(B(x,p^{-\ell})) \leq 1 \leq p^{2\ceil{\tau M_0}/\tau} p^{-2\ell/\tau} \lesssim p^{-2\ell/\tau} \ln(1 + p^{\ell}) g(p^{\ell}),  
$$
and we are done. 
Thus we can assume $\ceil{\tau M_{j-1}} < \ell \leq \ceil{\tau M_{j}}$ for some integer $j \geq 1$. 
Moreover, 
since $\mu$ is the weak limit of a subsequence of $(\mu_k)_{k=1}^{\infty}$ and $B(x,p^{-\ell})$ is both open and closed, we know $\mu(B(x,p^{-\ell}))$ is the limit of a subsequence of $(\mu_k(B(x,p^{-\ell})))_{k=1}^{\infty}$. 
Therefore, to prove \eqref{p kauf b g}, it suffices to prove 
\begin{lem}\label{reg thm k lemma}
For all $x \in \ZZ_p$ and $j,\ell \in \NN$ with $\ceil{\tau M_{j-1}} < \ell \leq \ceil{\tau M_{j}}$ there is a $k_0(x,j,\ell)>0$ such that 
\begin{align}\label{reg thm k}
\mu_k(B(x,p^{-\ell})) \lesssim p^{-2\ell/\tau} \ln(p^{\ell}) g(p^{\ell}) 
%\quad \forall x \in \ZZ_p, j,k,l \in \NN, \ceil{\tau M_{j-1}} < \ell \leq \ceil{\tau M_{j}}, k \geq k_0(x,j,\ell),
\end{align}
for all integers $k \geq k_0(x,j,\ell)$. 
%where $k_0(x,j,\ell)$ is a positive integer that may depend on $x,j,\ell$. 
%We will prove \eqref{reg thm k} with $k_0(x,j,\ell) = j$. 
\end{lem}
We will prove Lemma \ref{reg thm k lemma} with $k_0(x,j,\ell) = j$.

We introduce the following definitions. For $k \in \NN$, $P_k = F_{M_1} \cdots F_{M_k}$ and any ball of the form $B(r/q,p^{-\ceil{\tau M_k}})$ with $(q,r) \in Q_{M_k} \times R_{M_k}$ will be called a $k$-ball. 

We will need the following four lemmas.

\begin{lem}\label{disjoint balls}
If $(q,r),(q',r') \in Q_M \times R_M$ with $r/q \neq r'/q'$, then 
$$
\abs{\frac{r}{q} - \frac{r'}{q'}}_p > p^{-2M}.
$$
\end{lem}
\begin{proof} Since $|q|_p=|q'|_p=1$, $rq' \neq r'q$, and $0 \leq r,q,r',q' < p^M$, we have 
$$
\abs{\frac{r}{q} - \frac{r'}{q'}}_p = \abs{rq' - r'q}_p \geq \abs{rq'-r'q}^{-1} > p^{-2M}.
$$
\end{proof}

\begin{lem}\label{FM bound lemma}
For every $M \in \NN$, 
\begin{align}\label{FM bound 1}
F_M(x) \leq \frac{   p^{  \ceil{\tau M}   }  }{ |Q_{M}||R_{M}| }  \quad \forall x \in \ZZ_p, \; p^{-\ceil{\tau M}} < |x|_p  < 1.
\end{align}
\end{lem}
\begin{proof}
Fix $x \in \ZZ_p$ with $p^{-\ceil{\tau M}} < |x|_p  < 1$. 
Since 
$$
F_M(x) = \frac{1}{|Q_{M}||R_{M}|} \sum_{(q,r) \in Q_M \times R_M} p^{\ceil{\tau M}} \one_{B(r/q,p^{-\ceil{\tau M}})}(x), 
$$
it suffices to prove that the sum can have most one non-zero term. 
Thus, seeking a contradiction, suppose there are two pairs $(q,r) \neq (q',r')$ in $Q_M \times R_M$ 
such that $x \in B(r/q,p^{-\ceil{\tau M}}) \cap B(r'/q',p^{-\ceil{\tau M}})$. 
This implies $|r/q - r'/q'|_p \leq p^{-\ceil{\tau M}}$. 
Then Lemma \ref{disjoint balls} gives $r/q = r'/q'$. 
Since $(q,r) \neq (q',r')$, we must have $q \neq q'$. 
Then, because $q$ and $q'$ are primes, the number $r/q = r'/q'$ must be an integer.   
Furthermore, since 
$0 \leq r, r' < p^M$ and $\frac{1}{2}p^M \leq  q, q'$, 
we have either $r/q = r'/q' = 0$ or $r/q = r'/q' = 1$. 
Thus $x \in B(0,p^{-\ceil{\tau M}})$ or $x \in B(1,p^{-\ceil{\tau M}})$. 
Both possibilities contradict that $p^{-\ceil{\tau M}} < |x|_p  < 1$.
\end{proof}

\begin{lem}\label{non-i-ball lemma}
Let $x \in \ZZ_p$ and $j,\ell \in \NN$ with $\ell \leq \ceil{\tau M_{j}}$. Let $J$ be the number of $j$-balls that intersect $B(x,p^{-\ell})$. Then: 
\begin{description}
\item[(a)] $J \leq \max\cbr{1,p^{2M_j - \ell}}$
\item[(b)] $J \leq \max\cbr{1,p^{M_{j}-\ell} } |Q_M|$
\end{description}
\end{lem}
\begin{proof}
We prove (a) by considering two cases. 

{Case:} $\ell \geq 2M_j$. 
If two distinct $j$-balls $B(r/q,p^{-\ceil{\tau M_{j}}})$ and $B(r'/q',p^{-\ceil{\tau M_{j}}})$ intersect $B(x,p^{-\ell})$, then 
$\abs{r/q - r'/q'}_p \leq p^{-\ell},$ 
which contradicts Lemma \ref{disjoint balls}. Thus $J \leq 1$.

{Case:} $\ell < 2M_j$. Then $B(x,p^{-\ell})$ is a union of $p^{2M_j - \ell}$ balls of radius $p^{-2M_j}$. By Lemma \ref{disjoint balls}, any ball of radius $p^{-2M_j}$ intersects (hence contains) at most one $j$-ball. Thus $J \leq p^{2M_j-\ell}$.

Now we turn to the proof of (b). Suppose $(q,r) \in Q_{M_j} \times R_{M_j}$.  
Note that $B(x,p^{-\ell})$ intersects the $j$-ball $B(r/q,p^{-\ceil{\tau M_j}})$ if and only if 
$\abs{r/q - x}_p \leq p^{-\ell}$, which (because $|q|_p = 1$) is the case if and only if $r \equiv qx \pmod{p^{\ell}}$. 
Therefore $J$ is less than or equal to the number of $(q,r) \in Q_{M_j} \times R_{M_j}$ such that $r \equiv qx \pmod{p^{\ell}}$. 
The proof is completed by noting that, for any $q \in Q_{M_j}$ (in fact, for any $q \in \ZZ$), the number of integers $r$ with $r \equiv qx \pmod{p^{\ell}}$ and $0 \leq r < p^{M_j}$ is $\leq p^{M_j - \ell}$ if $M_j \geq \ell$ and is $\leq 1$ if $M_j \leq \ell$.
\end{proof}

\begin{lem}\label{i-ball lemma}
Let $j, k \in \NN$ with $j \leq k$.
If $B$ is a $j$-ball such that $B \cap \supp(P_k) \neq \emptyset$, then $B \cap \supp(P_k)$ is a union of 
at most 
%$K$ $k$-balls, where  
$$
%K \leq 
\prod_{i=j+1}^{k} |Q_{M_{i}}| p^{M_{i} - \ceil{\tau M_{i-1}}}  %.
$$
$k$-balls.
\end{lem}
\begin{proof}
Let $B$ be a $j$-ball such that $B \cap \supp(P_{k}) \neq \emptyset$. The proof is by induction on $k$. 

Base Step: $k = j$. Since $\supp(F_{M_j})$ is a union of $j$-balls, the same is true of $\supp(P_j)$. 
Since intersecting $j$-balls are equal, $B \cap \supp(P_j) = B$. 

Inductive Step: $k > j$. Note $B \cap \supp(P_{k})$ is the union of all $k$-balls contained in $B \cap \supp(P_{k-1})$. Since $\supp(P_{k}) \subseteq \supp(P_{k-1})$, we have $B \cap \supp(P_{k-1}) \neq \emptyset$. By the inductive hypothesis, $B \cap \supp(P_{k-1})$ is a union of 
%exactly 
at most 
$$
\prod_{i=j+1}^{k-1} |Q_{M_{i}}| p^{(M_{i} - \ceil{\tau M_{i-1}})}
$$
$(k-1)$-balls. Let $B(r'/q',p^{-\ceil{\tau M_{k-1}}})$ be any such $(k-1)$-ball. 
It suffices to show that $B(r'/q',p^{-\ceil{\tau M_{k-1}}})$ contains $\leq |Q_{M_{k}}| p^{(M_{k} - \ceil{\tau M_{k-1}})}$ $k$-balls. 
This follows from Lemma \ref{non-i-ball lemma}(b) by taking $\ell = \ceil{\tau M_{k-1}}$. 
\end{proof}

%
%
%The proofs of Lemmas \ref{FM bound lemma}, \ref{i-ball lemma}, and \ref{non-i-ball lemma} are given in Section \ref{proofs of lemmas}. 

%Now we are ready to prove \eqref{reg thm k}. 
Now we are ready to prove Lemma \ref{reg thm k lemma}, which (as we noted above) implies \eqref{p kauf b g}.

\begin{proof}[Proof of Lemma \ref{reg thm k lemma}]
Let $x \in \ZZ_p$ and let $j,k,l \in \NN$ with $\ceil{\tau M_{j-1}} < \ell \leq \ceil{\tau M_{j}}$ and $k \geq j$. 
%Let $k \in \NN$ with $j \leq k$. 
Let $B_1,\ldots,B_J$ be the collection of all $j$-balls that intersect $B(x,p^{-\ell})$. 
%By \eqref{ultra}, 
These balls are disjoint and contained in $B(x,p^{-\ell})$. 
Since $\supp(P_k) \subseteq \supp(P_j)$, and since $\supp(P_j)$ is a union of $j$-balls, we have
\begin{align*}
\mu_k(B(x,p^{-\ell}))  
= \sum_{i = 1}^{J} \mu_k(B_i) 
= \sum_{i = 1}^{J} \int_{B_i \cap \supp(P_k)} \psi_0(x) F_{M_1}(x) \cdots F_{M_k}(x) dx.
\end{align*}
%
%Using \eqref{psi0 support} and Lemma \ref{FM bound lemma}, we obtain 
%$$
%\mu_k(B(x,p^{-\ell})) \leq  \prod_{i=1}^{k} \frac{p^{\ceil{\tau M_i}}}{|Q_{M_i}||R_{M_i}|} \sum_{i = 1}^{J} \int_{B_i \cap \supp(\mu_k)} dx.
%$$
%%%Using Lemma \ref{FM bound lemma}, Lemma \ref{i-ball lemma}, and the fact that $k$-balls have Haar measure $p^{-\ceil{\tau M_k}}$, we obtain
%Then using Lemma \ref{i-ball lemma} and the fact that $k$-balls have Haar measure $p^{-\ceil{\tau M_k}}$, we obtain 
%
%$$
%\mu_k(B(x,p^{-\ell})) \leq \frac{J}{|Q_{M_j}||R_{M_j}|} \prod_{i=1}^{j-1} \frac{p^{\ceil{\tau M_i}}}{|Q_{M_i}||R_{M_i}|}.
%$$
%
First using \eqref{psi0 support}, \eqref{psi0 bound} and Lemma \ref{FM bound lemma}, and then using Lemma \ref{i-ball lemma}, $|R_{M}|=p^{M}$, and the fact that $k$-balls have Haar measure $p^{-\ceil{\tau M_k}}$, we obtain 
\begin{align*}
\mu_k(B(x,p^{-\ell})) 
&
\leq  \norm{\psi_0}_{\infty} \prod_{i=1}^{k} \frac{p^{\ceil{\tau M_i}}}{|Q_{M_i}||R_{M_i}|} \sum_{i = 1}^{J} \int_{B_i \cap \supp(P_k)} dx
\\
&
\leq  \norm{\psi_0}_{\infty} \frac{J}{|Q_{M_j}||R_{M_j}|} \prod_{i=1}^{j-1} \frac{p^{\ceil{\tau M_i}}}{|Q_{M_i}||R_{M_i}|}.
\end{align*}
Now we consider three cases and use \eqref{Mk size 3}, Lemma \ref{non-i-ball lemma}, $|Q_{M}| \approx p^{M} / \ln(p^{M})$, and $|R_{M}|=p^{M}$.

{Case:} $2M_j < \ell \leq \ceil{\tau M_{j}}$. 
We get 
$$
\mu_k(B(x,p^{-\ell})) 
\leq \norm{\psi_0}_{\infty} \frac{1}{|Q_{M_j}||R_{M_j}|} g(p^{M_j})
\approx p^{-2M_j} \ln(p^{M_j}) g(p^{M_j})
$$
Since $\ell \leq \ceil{\tau M_{j}} \leq 1 + \tau M_{j}$, we have $p^{-2M_j} \leq p^{2/\tau} p^{-2\ell/\tau} \lesssim p^{-2\ell/\tau}$. 
Thus \eqref{reg thm k} follows immediately.

{Case:} $M_j < \ell \leq 2M_j$. 
We get
$$
\mu_k(B(x,p^{-\ell})) 
\leq  
\norm{\psi_0}_{\infty}\frac{p^{2M_j-\ell}}{|Q_{M_j}||R_{M_j}|}  g(p^{M_{j}})
\approx
p^{-\ell} \ln(p^{M_j}) g(p^{M_{j}})
$$
%Since $\tau > 2$ and $M_{j} < \ell$, 
Since $\tau > 2$, we have $p^{-\ell} < p^{-2\ell/\tau}$. 
%it follows that 
%$$
%\mu_k(B(x,p^{-\ell})) \lesssim p^{-2\ell/\tau} \ln^2(p^{\ell}) \ln(p^{M_j}).
%$$
Thus \eqref{reg thm k} follows immediately.

{Case:} $\ceil{\tau M_{j-1}} < \ell \leq M_j$. 
We get 
\begin{align*}
\mu_k(B(x,p^{-\ell})) 
& 
\leq  
\norm{\psi_0}_{\infty} \frac{|Q_{M_j}| p^{M_j-\ell}}{|Q_{M_j}||R_{M_j}|} \cdot \frac{p^{\ceil{\tau M_{j-1}}}}{|Q_{M_{j-1}}||R_{M_{j-1}}|} g(p^{M_{j-1}})
\\ 
& 
\approx 
p^{\ceil{\tau M_{j-1}} - 2M_{j-1} - \ell} \ln(p^{M_{j-1}}) g(p^{M_{j-1}}). 
\end{align*}
Since $\tau > 2$ and $\tau M_{j-1} < \ell$, 
%and $\ceil{\tau M_{j-1}} - 2M_{j-1} \leq 1 + \tau M_{j-1}(1-2/\tau)$, 
we have 
$$
\ceil{\tau M_{j-1}} - 2M_{j-1} \leq 1 + \tau M_{j-1}\rbr{1 - \frac{2}{\tau}} \leq 1 + \ell - \frac{2\ell}{\tau}.
$$
Thus \eqref{reg thm k} follows immediately. 
%Since $\tau > 2$ and $\tau M_{j-1} < \ell$, \eqref{reg thm k} follows immediately.
%%%Since $\tau > 2$ and $\tau M_{j-1} < \ell$, we have $\ceil{\tau M_{j-1}} - 2M_{j-1} - 1 \leq \tau M_{j-1}(1 - 2/\tau) \leq \ell(1-2/\tau)$, and so
\end{proof}

%\subsection{Proofs of Lemmas}\label{proofs of lemmas}

\subsection{Proof of Lemma \ref{FM lemma}}\label{proofs of lemmas 1}

\begin{proof}%[Proof of Lemma \ref{FM lemma}]
Let $M \in \NN$ and $s \in \QQ_p/\ZZ_p$. 
For $|s|_p > p^{\ceil{\tau M}}$, Lemma \ref{phi lemma} implies \eqref{FM4}. 
For $|s|_p \leq p^{\ceil{\tau M}}$, Lemma \ref{phi lemma} gives 
\begin{align}\label{FM5}
\widehat{F_M}(s) 
= |Q_M|^{-1}|R_M|^{-1} 
\sum_{q \in Q_M} 
\sum_{0 \leq r < p^{M}} 
e( \cbr{ rs / q }_p  ).
\end{align}
Setting $s = 0$ yields \eqref{FM1}. 
From now on, assume $0 < |s|_p \leq p^{\ceil{\tau M}}$. 
So $|s|_p = p^{\ell}$ for some $\ell \in \cbr{1, \ldots, \ceil{\tau M}}$.  
We will study the sum over $r$ in \eqref{FM5}.   
Fix $q \in Q_M$.  
Since $|q|_p = 1$, we have $|s/q|_p=|s|_p=p^{\ell}$. 
Thus the $p$-adic expansion of $s/q$ has the form 
\begin{align}\label{s/q expansion}
\frac{s}{q} = \sum_{i=-\ell}^{\infty} c_i p^i, \quad c_i \in \cbr{0,1,\ldots,p-1}, \; c_{-\ell} \neq 0.
\end{align}
Evidently $0 < \cbr{ s / q }_p < 1$, and so $e(\cbr{ s / q }_p) \neq 1$. 
Because of \eqref{additive}, we have the geometric summation formula 
\begin{align}\label{geo-frac}
\sum_{0 \leq r < p^{M}} e( \cbr{rs/q}_p ) = \frac{1 - e(\cbr{s p^M / q}_p)}{1 - e(\cbr{s / q}_p)}.
\end{align}
If $|s|_p \leq p^M$, we have $\cbr{ s p^M / q }_p = 0$, hence the sum in \eqref{geo-frac} is zero. 
Applying this observation to \eqref{FM5} proves \eqref{FM2}.

Now only \eqref{FM3} remains to be proved. 
Assume $p^M < |s|_p = p^{\ell} \leq p^{\ceil{\tau M}}$. 
For all $z \in \RR$, 
$|1 - e(z)| =$ $2|\sin(\pi z)| =$ $2\sin(\pi \| z \|) \geq$ $\pi \| z \|$,  
where $\| z \| = \min_{k \in \ZZ}|z-k|$ is the distance from $z$ to the nearest integer. 
Hence the sum in \eqref{geo-frac} satisfies 
\begin{align}\label{geo-frac-inequality}
\abs{\sum_{0 \leq r < p^{M}} e( \cbr{rs/q}_p )} \leq \min\cbr{\frac{1}{\| \cbr{s / q}_p \|},p^M}.
\end{align}
In light of \eqref{s/q expansion}, 
$$
\| \cbr{s / q}_p \| 
=
\left\{
\begin{array}{ll}
\cbr{s / q}_p  = \sum_{i=-\ell}^{-1} c_i p^i & \text{if } \cbr{s / q}_p \leq 1/2 \\
1 - \cbr{s / q}_p  = 1 - \sum_{i=-\ell}^{-1} c_i p^i & \text{if } \cbr{s / q}_p  %\geq 1/2.
> 1/2.
\end{array} \right.
$$
Combining \eqref{FM5}, \eqref{geo-frac-inequality}, and the fact that $p^{-\ell} \leq \| \cbr{s / q}_p \| < 1$ leads to
\begin{align}\label{min-2}
|\widehat{F_M}(s)| \leq |Q_M|^{-1} |R_M|^{-1} \sum_{k=1}^{\ell} \sum_{\substack{\frac{1}{2}p^{M} \leq q < p^{M} \\ |q|_p = 1, \; q \; \text{prime} \\ p^{-k} \leq \| \cbr{s / q}_p \| < p^{-k+1} }} \min\cbr{p^k,p^M}.
\end{align}
For fixed $1 \leq k \leq \ell$, we now estimate the number of terms in the sum over $q$ in \eqref{min-2}. This estimate is similar to the proof of Theorem 1 in Cilleruelo and Garaev's paper \cite{cill-gar}. Consider any prime $q$ with $\frac{1}{2}p^{M} \leq q < p^{M}$, $|q|_p = 1$, and $p^{-k} \leq \| \cbr{s / q}_p \| < p^{-k+1}$. Define $N = \| \cbr{s/q}_p \| p^{\ell} q$. Note that $N$ is a positive integer $\leq p^{M+\ell-k+1}$. 
If $\cbr{s / q}_p \leq 1/2$, then 
$
N = \rbr{s/q - \sbr{s/q}_p} p^{\ell} q \equiv s p^{\ell} \pmod{p^{\ell}}.
$
Similarly, if $\cbr{s / q}_p > 1/2$, then 
$
N = \rbr{1 - s/q + \sbr{s/q}_p} p^{\ell} q \equiv -s p^{\ell} \pmod{p^{\ell}}.
$ 
Therefore $q$ is a prime $\geq \frac{1}{2}p^M$ that divides a positive integer $N$ with $N \leq p^{M+\ell-k+1}$ and $N \equiv \pm s p^{\ell} \pmod{p^{\ell}}$. 
The number of positive integers $N$ with $N \leq p^{M+\ell-k+1}$ and $N \equiv \pm s p^{\ell} \pmod{p^{\ell}}$ is 
%$\leq 2 \max\cbr{p^{M-k+1},1}$. 
$\lesssim \max\cbr{p^{M-k+1},1}$.
And the number of primes $q \geq \frac{1}{2}p^M$ that divide a given positive integer $N$ is 
%$\leq 2 \ln N / \ln p^M$. 
$\lesssim \ln N / \ln p^M$. 
Therefore the number of terms in the sum over $q$ in \eqref{min-2} is 
$$
%\leq 4 
\lesssim 
\max\cbr{p^{M-k+1},1} \frac{\ln p^{M+\ell-k+1}}{\ln p^M}.
$$
Thus \eqref{min-2} implies
\begin{align*}
|\widehat{F_M}(s)| 
%&\leq 4
&\lesssim 
|Q_M|^{-1} |R_M|^{-1} \sum_{k=1}^{\ell} \min\cbr{p^k,p^M} \max\cbr{p^{M-k+1},1} \frac{\ln p^{M+\ell-k+1}}{\ln p^M}.
\end{align*}
%Since $p^M < |s|_p = p^\ell \leq p^{\ceil{\tau M}}$, $|Q_M| \geq \frac{1}{4} p^M / \ln p^M$, and $|R_M| = p^M$, we obtain \eqref{FM3}.
Since $p^M < |s|_p = p^\ell \leq p^{\ceil{\tau M}}$, $|Q_M| \gtrsim p^M / \ln p^M$, and $|R_M| = p^M$, we obtain \eqref{FM3}. 
\end{proof}

\subsection{Proof of Lemma \ref{mu-k lemma}}\label{proofs of lemmas 2}

\begin{proof}%[Proof of Lemma \ref{mu-k lemma}]
Let $s \in \QQ_p / \ZZ_p$. 
The proof is by induction on $k$. 
The case $k=0$ follows immediately from \eqref{psi0 normal} and the definition $d\mu_0 = d\mu_{-1} = \psi_0 dx$.
Assume $k \geq 1$. 
The inductive hypothesis is that \eqref{mu-k 1}-\eqref{mu-k 4} hold with $k$ replaced by $k-1$. 
%By the Fourier inversion theorem (see \cite[p.102]{Folland-abstract-harmonic} or \cite[p.120]{Taibelson}) and Fubini's theorem, we have 
By the usual argument with the Fourier inversion theorem (see \cite[p.102]{folland-book-abstract} or \cite[p.120]{T75}) and Fubini's theorem, we have 
\begin{align}\label{convol}
\widehat{\mu_k}(s) = \widehat{F_{M_k} \mu_{k-1}}(s) = \sum_{t \in \QQ_p/\ZZ_p} \widehat{F_{M_k}}(s-t) \widehat{\mu_{k-1}}(t). 
\end{align}
%By \eqref{ultra}, \eqref{Mk size 1}, Lemma \ref{FM lemma}, and the inductive hypothesis, 
%each $t \in \QQ_p/\ZZ_p$ with $\widehat{F_{M_k}}(s-t) \widehat{\mu_{k-1}}(t) \neq 0$ 
%must satisfy either 
%$$
%|t|_p \leq p^{\ceil{\tau M_{k-1}}}, \quad t=s,
%$$
%or
%$$
%|t|_p \leq p^{\ceil{\tau M_{k-1}}}, \quad p^{M_k} < |s-t|_p \leq p^{\ceil{\tau M_k}}, \quad |s|_p = |s-t|_p.
%$$
If the summand $\widehat{F_{M_k}}(s-t) \widehat{\mu_{k-1}}(t)$ is non-zero, then we must have $|t|_p \leq p^{\ceil{\tau M_{k-1}}}$ by the inductive hypothesis, and either 
$$t=s \qquad \text{ or } \qquad p^{M_k} < |s|_p = |s-t|_p \leq p^{\ceil{\tau M_k}}$$ 
%%%$t=s$ or $p^{M_k} < |s|_p = |s-t|_p \leq p^{\ceil{\tau M_k}}$ 
%%by \eqref{ultra}, \eqref{Mk size 1}, and Lemma \ref{FM lemma}. 
by \eqref{Mk size 1} and Lemma \ref{FM lemma}. 
Therefore, if $|s|_p > p^{\ceil{\tau M_k}}$, every term of the sum in \eqref{convol} is zero, and $\widehat{\mu_k}(s) = 0$. This proves \eqref{mu-k 4}. 
On the other hand, if $|s|_p \leq p^{M_k}$, then only the $t=s$ term contributes to the sum, and $\widehat{\mu_k}(s) = \widehat{F_{M_k}}(0) \widehat{\mu_{k-1}}(s) = \widehat{\mu_{k-1}}(s)$. 
This proves \eqref{mu-k 2} and, using the inductive hypothesis, \eqref{mu-k 1}. 
Only \eqref{mu-k 3} remains to be proved. 
Suppose $p^{M_k} < |s|_p \leq p^{\ceil{\tau M_k}}$. 
%Then 
For all $t \in \QQ_p/\ZZ_p$ with 
$\widehat{F_{M_k}}(s-t) \widehat{\mu_{k-1}}(t) \neq 0$ 
we must have $|s|_p = |s-t|_p$, 
and so \eqref{FM3} gives 
$|\widehat{F_{M_k}}(s-t)| 
%\lesssim p^{-M_k} \ln^2(p^{M_k})
%\leq 32p^{1+1/\tau} |s|^{-1/\tau} \ln^2(|s|_p)
\lesssim |s|^{-1/\tau} \ln^2(|s|_p)
$. 
By the inductive hypothesis, 
$|\widehat{\mu_{k-1}}(t)| \leq \widehat{\mu_{k-1}}(0) = 1$ for all $t \in \QQ_p/\ZZ_p$. 
By counting digits, 
the number of $t \in \QQ_p / \ZZ_p$ with $|t|_p \leq p^{\ceil{\tau M_{k-1}}}$ is exactly $p^{\ceil{\tau M_{k-1}}}$;
hence, the sum in \eqref{convol} has at most $p^{\ceil{\tau M_{k-1}}}$ non-zero terms. 
Putting it all together, we get 
\begin{align*}
|\widehat{\mu_k}(s)| 
\lesssim  p^{\ceil{\tau M_{k-1}}} |s|_p^{-1/\tau} \ln^2(|s|_p).
\end{align*}
Finally, applying \eqref{Mk size 2} gives \eqref{mu-k 3}.
\end{proof}

\section{Proof of Theorem \ref{thm 2}}\label{proof thm 2}

\subsection{Reduction}

To prove Theorem \ref{thm 2}, it suffices to prove the seemingly weaker 
\begin{thm}\label{thm 2a}
Let $g$ be a non-negative non-decreasing function defined on $(0,\infty)$ such that $\lim_{x \rightarrow \infty} g(x)=\infty$. 
For every $\tau>(m+n)/m$, there exists a Borel probability measure $\mu$ supported on $W(m,n,\tau)$ such that 
\begin{align*}%%%\label{p kauf a mn g}
|\widehat{\mu}(\xi)| &\lesssim |\xi|_p^{-n/\tau} \ln^{n+1}(1+|\xi|_p)g(|\xi|_p) \quad \forall \xi \in \QQ_p^{mn}, \xi \neq 0.
\end{align*}
We emphasize that the constant implied by $\lesssim$ does not depend on $g$.
\end{thm}
The proof that Theorem \ref{thm 2a} implies Theorem \ref{thm 2} is analogous to the proof in Section \ref{reduction} that Theorem \ref{thm 1a} implies Theorem \ref{thm 1}.

\subsection{Proof of Theorem \ref{thm 2a}}\label{proof of 2a}

Let $\tau > (m+n)/m$. Let $g$ be any non-negative non-decreasing function defined on $(0,\infty)$ such that $\lim_{x \rightarrow \infty} g(x)=\infty$. 
For each $M \in \NN$, define $Q_M$ and $R_M$ as in Section \ref{proof of 1a}. Then
\begin{align*}
Q_M^n &= \cbr{q \in \ZZ^n : \frac{1}{2}p^{M} \leq q_j < p^M, \; |q_j|_p = 1, \; q_j \; \text{prime} \;\; \forall 1 \leq j \leq n }, \\
R_M^m &= \cbr{r \in \ZZ^m : 0 \leq r_i < p^M \;\; \forall 1 \leq i \leq m}.
\end{align*}
Note that $Q_M$ is non-empty unless $p=2$ and $M=1$. For everything that follows, we make the standing assumption that $M \geq 2$ if $p=2$. 
For each $q \in Q_M^n$ and $r \in R_M^m$, define the function $\phi_{q,r}$ on $\ZZ_{p}^{mn}$ by 
$$
\phi_{q,r}(x) = p^{m\ceil{\tau M}} \one_{B(0,1)}(p^{-\ceil{\tau M}}(xq-r)) \quad \forall x \in \ZZ_p^{mn}.
$$
For each $M \in \NN$, define the function $F_M$ on $\ZZ_p^{mn}$ by 
$$
F_M(x) 
= |Q_M^n|^{-1}|R_M^m|^{-1} \sum_{q \in Q_M^n} \sum_{r \in R_M^m} \phi_{q,r}(x)
\quad \forall x \in \ZZ_p^{mn}.
$$
Choose a strictly increasing sequence of non-negative integers $(M_k)_{k=0}^{\infty}$ such that for all $k \in \NN$ 
\begin{align}
\label{Mk size mn 1} 
\ceil{\tau M_{k-1}} &< M_k, \\
\label{Mk size mn 2} 
p^{mn\ceil{\tau M_{k-1}}} &< g(p^{M_k}). 
\end{align}
Let $\psi_0$ be any non-negative function on $\ZZ_p^{mn}$ such that
\begin{align}
\label{psi0 normal mn}
\widehat{\psi_0}(0) &=1, \\
\label{psi0 decay mn}
\widehat{\psi_0}(s) &= 0 \quad \text{ for all } s \in (\QQ_p/\ZZ_p)^{mn} \text{ with } |s|_p > p^{\ceil{\tau M_0}}.
\end{align}
In light of Lemma \ref{ball lemma}, we may choose, for example, 
$\psi_0 = \one_{B(0,1)}$.  
For each $k \in \NN$, define the measure $\mu_k$ on $\ZZ_p^{mn}$ by  
$$
d \mu_k(x) = \psi_0(x) F_{M_1}(x) \cdots F_{M_k}(x)  dx.
$$ 
For notational convenience in Lemma \ref{mu-k lemma} below, 
we define $d\mu_{-1}(x) = d\mu_0(x) = \psi_0(x) dx$. 
For each $s \in (\QQ_p / \ZZ_p)^{mn}$, define
$$
D(s) = \cbr{q \in \ZZ^n: \cbr{s_{ij}/q_j}_p = \cbr{s_{ij'}/q_{j'}}_p \; \forall 1 \leq i \leq m, \; 1 \leq j, j' \leq n }.
$$

The proof proceeds by the following sequence of lemmas.

\begin{lem}\label{phi lemma mn}
For all $M \in \NN$, $q \in Q_M^n$, $r \in R_M^m$, and $s \in (\QQ_p / \ZZ_p)^{mn}$, 
$$
\widehat{\phi_{q,r}}(s) = \left\{
\begin{array}{cl}
e( \sum_{i=1}^{m} \cbr{ r_i s_{i1} / q_{1} }_p  ) & \text{if } |s|_p \leq p^{\ceil{\tau M}} \text{ and } q \in D(s) \\
0 & \text{otherwise }
\end{array} \right.
$$
\end{lem}

\begin{lem}\label{FM lemma mn}
For all $M \in \NN$ and $s \in (\QQ_p / \ZZ_p)^{mn}$, 
\begin{IEEEeqnarray}{rCll}
\label{FM1 mn}
\widehat{F_M}(s) & = & 1 & \quad \text{if } s  =  0 \\
\label{FM2 mn}
\widehat{F_M}(s) & = & 0 & \quad \text{if } 0  <  |s|_p  \leq  p^M \\
\label{FM3 mn}
|\widehat{F_M}(s)| 
&\lesssim & |s|^{-n/\tau} \ln^{n+1} (|s|_p)   
%&\lesssim p^{-nM}  \ln^{n+1} (p^{M})  
& \quad \text{if }  p^M  <  |s|_p  \leq  p^{\ceil{\tau M}} \\
\label{FM4 mn}
\widehat{F_M}(s) & = & 0 & \quad \text{if }  |s|_p  >  p^{\ceil{\tau M}}
\end{IEEEeqnarray}
\end{lem}

\begin{lem}\label{mu-k lemma mn}
For all integers $k \geq 0$ and all $s \in (\QQ_p / \ZZ_p)^{mn}$, 
\begin{IEEEeqnarray}{rCll}
\label{mu-k 1 mn}
\widehat{\mu_k}(s) & = & 1 & \quad \text{if }  s  =  0 \\
\label{mu-k 2 mn}
\widehat{\mu_k}(s) & = & \widehat{\mu_{k-1}}(s) &  \quad \text{if }  0  <  |s|_p  \leq  p^{M_k} \\
\label{mu-k 3 mn}
|\widehat{\mu_k}(s)| 
&\lesssim & |s|^{-n/\tau} \ln^{n+1} (|s|_p) g(|s|_p)
%&\lesssim p^{-M_k}  \ln^{n+2} (p^{M_k})  
& \quad \text{if }  p^{M_k}  <  |s|_p  \leq  p^{\ceil{\tau M_k}} \\
\label{mu-k 4 mn}
\widehat{\mu_k}(s) & = & 0 & \quad \text{if }  |s|_p  >  p^{\ceil{\tau M_k}}
\end{IEEEeqnarray}
\end{lem}
Unlike Lemma \ref{phi lemma}, Lemma \ref{phi lemma mn} is not quite an immediate corollary of Lemma \ref{ball lemma}. 
The proof of Lemma \ref{FM lemma mn} is a generalization of the proof of Lemma \ref{FM lemma}. 
The proofs of Lemmas \ref{phi lemma mn} and \ref{FM lemma mn} are given in 
%Section \ref{proofs of lemmas mn}. 
Sections \ref{proofs of lemmas mn 1} and \ref{proofs of lemmas mn 2}, respectively. 
We omit the proof of Lemma \ref{mu-k lemma mn} because it is virtually identical to the proof of Lemma \ref{mu-k lemma} in Section \ref{proofs of lemmas 2}.

The rest of the proof of Theorem \ref{thm 2a} proceeds as in Section \ref{proof of 1a}, so we omit it.

%
%
%

%\subsection{Proofs of Lemmas}\label{proofs of lemmas mn}

\subsection{Proof of Lemma \ref{phi lemma mn}}\label{proofs of lemmas mn 1}

%\begin{proof}[Proof of Lemma \ref{phi lemma mn}]
\begin{proof}
Let $M \in \NN$, $q \in Q_M^n$, $r \in R_M^m$, and $s \in \QQ_p/\ZZ_p$ be given. 
Define the function $\phi_r$ on $\ZZ_p^m$ by 
$$
\phi_r(x) = p^{m\ceil{\tau M}} \one_{B(0,1)}(p^{-\ceil{\tau M}}(x-r)) \quad \forall x \in \ZZ_p^m.
$$
By Lemma \ref{ball lemma}, for all $k \in (\QQ_p/\ZZ_p)^m$,  
\begin{align}\label{phi r eq}
\widehat{\phi_r}(k) = 
\left\{
\begin{array}{cl}
e( \cbr{ r \cdot k }_p  ) & \text{if } |k|_p \leq p^{\ceil{\tau M}} \\
0 & \text{if } |k|_p > p^{\ceil{\tau M}}
\end{array} \right.
\end{align}
By Fourier inversion (see \cite[p.102]{folland-book-abstract} or \cite[p.120]{T75}), 
$$
\phi_r(x) = \sum_{k \in (\QQ_p / \ZZ_p)^m} \widehat{\phi_r}(k) e(-\cbr{k \cdot x}_p) \quad \forall x \in \ZZ_p^m.
$$
Therefore, since $|q_j|_p = 1$ for all $1 \leq j \leq n$, 
$$
\phi_{q,r}(x) = \phi_{r}(xq) = \sum_{k \in (\QQ_p / \ZZ_p)^m} \widehat{\phi_r}(k) e(-\cbr{k \cdot xq}_p) \quad \forall x \in \ZZ_p^{mn}.
$$
By Fubini's theorem, 
\begin{align}\label{phi qr eq 1}
\widehat{\phi_{q,r}}(s)
&=
\sum_{k \in (\QQ_p / \ZZ_p)^m} \widehat{\phi_r}(k) \int_{\ZZ_p^{mn}} e(\cbr{s \cdot x}_p) e(-\cbr{k \cdot xq}_p)  dx \\
\notag
&= 
\sum_{k \in (\QQ_p / \ZZ_p)^m} \widehat{\phi_r}(k) \prod_{i=1}^{m} \prod_{j=1}^{n} \int_{\ZZ_p} e(\cbr{x_{ij}(s_{ij} - k_i q_j)}_p) dx_{ij}.
\end{align}
Fix $k \in (\QQ_p / \ZZ_p)^m$. 
By Lemma \ref{ball lemma}, 
$$
\int_{\ZZ_p} e(\cbr{x_{ij}(s_{ij} - k_i q_j)}_p) dx_{ij}
=
\left\{
\begin{array}{cl}
1 & \text{if } |s_{ij} - k_i q_j|_p \leq 1 \\
0 & \text{otherwise }
\end{array} \right.
$$
Note that, since $k_i \in \QQ_p/\ZZ_p$ and $|q_j|_p=1$, $|s_{ij} - k_i q_j|_p \leq 1$ is equivalent to $k_i = \cbr{s_{ij}/q_j}_p$. 
Thus \eqref{phi qr eq 1} gives
$$
\widehat{\phi_{q,r}}(s) = 
\left\{
\begin{array}{cl}
\widehat{\phi_r}(\cbr{s_{11}/q_1}_p, \ldots, \cbr{s_{m1}/q_1}_p) & \text{if } q \in D(s) \\
0 & \text{otherwise }
\end{array} \right.
$$
To complete the proof, use \eqref{additive}, \eqref{phi r eq}, and the fact that for all $\ell \geq 0$ and $y \in \QQ_p$, we have $|y|_p \leq p^{\ell}$ if and only if $|\cbr{y}_p|_p \leq p^{\ell}$.
\end{proof}

\subsection{Proof of Lemma \ref{FM lemma mn}}\label{proofs of lemmas mn 2}

%\begin{proof}[Proof of Lemma \ref{FM lemma mn}] 
\begin{proof}
Let $M \in \NN$ and $s \in (\QQ_p/\ZZ_p)^{mn}$. 
Choose $1 \leq i_0 \leq m$ and $1 \leq j_0 \leq n$ such that $|s_{i_0 j_0}|_p = |s|_p$. 
For $|s|_p > p^{\ceil{\tau M}}$, Lemma \ref{phi lemma mn} implies \eqref{FM4 mn}. 
For $|s|_p \leq p^{\ceil{\tau M}}$, \eqref{additive}, Lemma \ref{phi lemma mn}, and the definition of $D(s)$ give 
\begin{align}\label{FM5 mn}
\widehat{F_M}(s) 
= |Q_M^n|^{-1}|R_M^m|^{-1} 
\sum_{q \in Q_M^n \cap D(s)} 
%
%\rbr{\sum_{0 \leq r_1 < p^M } e( \cbr{ r_1 s_{1j_0} / q_{j_0} }_p  )}
%\cdots
%\rbr{\sum_{0 \leq r_m < p^M } e( \cbr{ r_m s_{mj_0} / q_{j_0} }_p  )}.
\prod_{i=1}^{m}
\sum_{0 \leq r_i < p^M } e( \cbr{ r_i s_{ij_0} / q_{j_0} }_p  )
\end{align}
Setting $s = 0$ yields \eqref{FM1 mn}. 

From now on, assume $0 < |s|_p \leq p^{\ceil{\tau M}}$. 
So $|s|_p = p^{\ell}$ for some $\ell \in \cbr{1, \ldots, \ceil{\tau M}}$.  
We will study the sum over $r_{i_0}$ in \eqref{FM5}.   
Fix $q \in Q_M^n \cap D(s)$.  
Since $|q_{j_0}|_p = 1$, we have $|s_{i_0j_0}/q_{j_0}|_p = |s_{i_0j_0}|_p = |s|_p = p^{\ell}$. 
Thus the $p$-adic expansion of $s_{i_0 j_0}/q_{j_0}$ has the form 
\begin{align}\label{s/q expansion mn}
\frac{s_{i_0 j_0}}{q_{j_0}} = \sum_{i=-\ell}^{\infty} c_i p^i, \quad c_i \in \cbr{0,1,\ldots,p-1}, \; c_{-\ell} \neq 0.
\end{align}
Evidently $0 < \cbr{ s_{i_0 j_0} / q_{j_0} }_p < 1$, and so $e(\cbr{ s_{i_0 j_0} / q_{j_0} }_p) \neq 1$. 
Because of \eqref{additive}, we have the geometric summation formula 
\begin{align}\label{geo-frac mn}
\sum_{0 \leq r_{i_0} < p^{M}} e( \cbr{r_{i_0} s_{i_0 j_0} / q_{j_0}}_p ) = \frac{1 - e(\cbr{p^M s_{i_0 j_0} / q_{j_0}}_p)}{1 - e(\cbr{s_{i_0 j_0} / q_{j_0}}_p)}.
\end{align}
If $|s|_p \leq p^M$, we have $\cbr{ p^M s_{i_0 j_0} / q_{j_0} }_p = 0$; hence, the sum in \eqref{geo-frac mn} is zero. 
Applying this observation to \eqref{FM5 mn} proves \eqref{FM2 mn}.

Now only \eqref{FM3 mn} remains to be proved. 
Assume $p^M < |s|_p = p^{\ell} \leq p^{\ceil{\tau M}}$. 
For all $z \in \RR$, 
$|1 - e(z)| =$ $2|\sin(\pi z)| =$ $2\sin(\pi \| z \|) \geq$ $\pi \| z \|$,  
where $\| z \| = \min_{k \in \ZZ}|z-k|$ is the distance from $z$ to the nearest integer. 
Hence the sum in \eqref{geo-frac mn} satisfies 
\begin{align}\label{geo-frac-inequality mn}
\abs{\sum_{0 \leq r_{i_0} < p^{M}} e( \cbr{r_{i_0} s_{i_0 j_0} / q_{j_0}}_p )} \leq \min\cbr{\frac{1}{\| \cbr{s_{i_0 j_0} / q_{j_0}}_p \|},p^M}.
\end{align}
We will also need that 
\begin{align}\label{triv geo-frac-inequality mn}
\abs{\sum_{0 \leq r_{i} < p^{M}} e( \cbr{r_{i} s_{i j_0} / q_{j_0}}_p )} \leq p^M \quad \forall 1 \leq i \leq m.
\end{align}
In light of \eqref{s/q expansion mn}, 
$$
\| \cbr{s_{i_0 j_0} / q_{j_0}}_p \| 
=
\left\{
\begin{array}{ll}
\cbr{s_{i_0 j_0} / q_{j_0}}_p  = \sum_{i=-\ell}^{-1} c_i p^i & \text{if } \cbr{s_{i_0 j_0} / q_{j_0}}_p \leq 1/2 \\
1 - \cbr{s_{i_0 j_0} / q_{j_0}}_p  = 1 - \sum_{i=-\ell}^{-1} c_i p^i & \text{if } \cbr{s_{i_0 j_0} / q_{j_0}}_p  %\geq 1/2.
> 1/2.
\end{array} \right.
$$
Combining \eqref{FM5 mn}, \eqref{geo-frac-inequality mn} \eqref{triv geo-frac-inequality mn}, and the fact that $p^{-\ell} \leq \| \cbr{s_{i_0 j_0} / q_{j_0}}_p \| < 1$ leads to
\begin{align}\label{min-2 mn}
|\widehat{F_M}(s)| 
\leq 
|Q_M^n|^{-1} |R_M^m|^{-1} \sum_{k=1}^{\ell} 
%\sum_{\substack{q \in Q_M^n \cap D(s) \\ p^{-k} \leq \| \cbr{s_{i_0 j_0} / q_{j_0}}_p \| < p^{-k+1} }}    p^{(m-1)M} \min\cbr{p^k,p^M}.
\sum_q p^{(m-1)M} \min\cbr{p^k,p^M}, 
\end{align}
where the inner sum runs over all $q \in Q_M^n \cap D(s)$ such that 
$p^{-k} \leq \| \cbr{s_{i_0 j_0} / q_{j_0}}_p \| < p^{-k+1}$. 
We claim $(q_1,\ldots,q_n) \mapsto q_{j_0}$ is an injection from $Q_M^n \cap D(s)$ to the set
$$
\cbr{q_{j_0} \in \ZZ : \frac{1}{2}p^M \leq q_{j_0} < p^M, \; |q_{j_0}|_p = 1, \; q_{j_0} \; \text{prime} }.
$$
This claim follows from the following two observations. First, for each $q \in Q_M^n$ and $1 \leq j \leq n$, we have 
$\cbr{s_{i_0 j_0} / q_{j_0}}_p = \cbr{s_{i_0 j} / q_{j}}_p$
if and only if $|s_{i_0 j_0} / q_{j_0} - s_{i_0 j} / q_{j}|_p \leq 1$ if and only if $|q_j - q_{j_0} s_{i_0 j} s_{i_0 j_0}^{-1}|_p \leq |s_{i_0 j_0}^{-1}|_p = p^{-\ell}$ if and only if $q_j \equiv q_{j_0} s_{i_0 j}s_{i_0 j_0}^{-1} \pmod{p^{\ell}}$.
Second, for any given $b \in \QQ_p$, there can be at most one integer $a$ satisfying $a \equiv b \pmod{p^{\ell}}$ and $\frac{1}{2}p^M \leq a < p^M \leq p^{\ell}$. 
Applying the claim to \eqref{min-2 mn} yields
\begin{align}\label{min-3 mn}
|\widehat{F_M}(s)| 
\leq 
|Q_M^n|^{-1} |R_M^m|^{-1} \sum_{k=1}^{\ell} 
%\sum_{\substack{\frac{1}{2}p^{M} \leq q_{j_0} < p^{M} \\ |q_{j_0}|_p = 1, \; q_{j_0} \; \text{prime} \\ p^{-k} \leq \| \cbr{s_{i_0 j_0} / q_{j_0}}_p \| < p^{-k+1} }} 
%p^{(m-1)M} \min\cbr{p^k,p^M}.
\sum_{q_{j_0}} p^{(m-1)M} \min\cbr{p^k,p^M}, 
\end{align}
where the inner sum runs over all $q_{j_0} \in Q_M$ such that $p^{-k} \leq \| \cbr{s_{i_0 j_0} / q_{j_0}}_p \| < p^{-k+1}$.  
Arguing as in the proof of Lemma \ref{FM lemma} in 
Section \ref{proofs of lemmas 1}, 
we see that for each fixed $1 \leq k \leq \ell$ the number of terms in the sum over $q_{j_0}$ in \eqref{min-3 mn} is
$$
\lesssim \max\cbr{p^{M-k+1},1} \frac{\ln p^{M+\ell-k+1}}{\ln p^M}.
$$
Thus \eqref{min-3 mn} implies
\begin{align*}
|\widehat{F_M}(s)| 
&\lesssim 
|Q_M^n|^{-1} |R_M^m|^{-1} \sum_{k=1}^{\ell} p^{(m-1)M} \min\cbr{p^k,p^M} \max\cbr{p^{M-k+1},1} \frac{\ln p^{M+\ell-k+1}}{\ln p^M}.
\end{align*}
Since $p^M < |s|_p = p^\ell \leq p^{\ceil{\tau M}}$, $|Q_M^n| \approx p^{nM} / (\ln p^M)^n$, and $|R_M^m| = p^{mM}$, we obtain \eqref{FM3 mn}.
\end{proof}

\ack This work was supported by NSERC.

%\bibliographystyle{plain}
%\bibliography{p-Adic_Salem}

%\bibliographystyle{amsplain}
%\bibliography{p-Adic_Salem}

\begin{thebibliography}{10}
\bibitem{Abercrombie}
A.~G. Abercrombie.
\newblock The {H}ausdorff dimension of some exceptional sets of {$p$}-adic
  integer matrices.
\newblock {\em J. Number Theory}, 53(2):311--341, 1995.

\bibitem{allen-troscheit}
D.~{Allen} and S.~{Troscheit}.
\newblock {The Mass Transference Principle: Ten Years On}.
\newblock \url{https://arxiv.org/abs/1704.06628v2}.

\bibitem{bak-seeger}
J.-G. Bak and A.~Seeger.
\newblock Extensions of the {S}tein-{T}omas theorem.
\newblock {\em Math. Res. Lett.}, 18(4):767--781, 2011.

\bibitem{beresnevich-bernik-dodson-velani}
V.~Beresnevich, V.~Bernik, M.~Dodson, and S.~Velani.
\newblock Classical metric {D}iophantine approximation revisited.
\newblock In {\em Analytic number theory}, pages 38--61. Cambridge Univ. Press,
  Cambridge, 2009.

\bibitem{beresnevich-dickinson-velani}
V.~Beresnevich, D.~Dickinson, and S.~Velani.
\newblock Measure theoretic laws for lim sup sets.
\newblock {\em Mem. Amer. Math. Soc.}, 179(846):x+91, 2006.

\bibitem{bes}
A.~S. Besicovitch.
\newblock Sets of {F}ractional {D}imensions ({IV}): {O}n {R}ational
  {A}pproximation to {R}eal {N}umbers.
\newblock {\em J. London Math. Soc.}, S1-9(2):126, 1934.

\bibitem{bluhm-1}
C.~Bluhm.
\newblock Random recursive construction of {S}alem sets.
\newblock {\em Ark. Mat.}, 34(1):51--63, 1996.

\bibitem{bluhm-2}
C.~Bluhm.
\newblock On a theorem of {K}aufman: {C}antor-type construction of linear
  fractal {S}alem sets.
\newblock {\em Ark. Mat.}, 36(2):307--316, 1998.

\bibitem{bogachev}
V.~I. Bogachev.
\newblock {\em Measure theory. {V}ol. {I}, {II}}.
\newblock Springer-Verlag, Berlin, 2007.

\bibitem{bovey-dodson}
J.~D. Bovey and M.~M. Dodson.
\newblock The {H}ausdorff dimension of systems of linear forms.
\newblock {\em Acta Arith.}, 45(4):337--358, 1986.

\bibitem{Carbery-Seeger-Waigner-Wright}
A.~Carbery, A.~Seeger, S.~Wainger, and J.~Wright.
\newblock Classes of singular integral operators along variable lines.
\newblock {\em J. Geom. Anal.}, 9(4):583--605, 1999.

\bibitem{chen-seeger}
X.~Chen and A.~Seeger.
\newblock Convolution powers of {S}alem measures with applications.
\newblock {\em Canad. J. Math.}, 69(2):284--320, 2017.

\bibitem{cill-gar}
J.~Cilleruelo and M.~Z. Garaev.
\newblock Concentration of points on two and three dimensional modular
  hyperbolas and applications.
\newblock {\em Geom. Funct. Anal.}, 21(4):892--904, 2011.

\bibitem{eggleston}
H.~G. Eggleston.
\newblock Sets of fractional dimensions which occur in some problems of number
  theory.
\newblock {\em Proc. London Math. Soc. (2)}, 54:42--93, 1952.

\bibitem{ekstrom}
F.~Ekstr\"om.
\newblock Fourier dimension of random images.
\newblock {\em Ark. Mat.}, 54(2):455--471, 2016.

\bibitem{EPS}
F.~Ekstr\"om, T.~Persson, and J.~Schmeling.
\newblock On the {F}ourier dimension and a modification.
\newblock {\em J. Fractal Geom.}, 2(3):309--337, 2015.

\bibitem{falconer-book-1}
K.~J. Falconer.
\newblock {\em The geometry of fractal sets}, volume~85 of {\em Cambridge
  Tracts in Mathematics}.
\newblock Cambridge University Press, Cambridge, 1986.

\bibitem{folland-book-abstract}
G.~B. Folland.
\newblock {\em A course in abstract harmonic analysis}.
\newblock Studies in Advanced Mathematics. CRC Press, Boca Raton, FL, 1995.

\bibitem{FOS}
J.~M. Fraser, T.~Orponen, and T.~Sahlsten.
\newblock On {F}ourier analytic properties of graphs.
\newblock {\em Int. Math. Res. Not. IMRN}, (10):2730--2745, 2014.

\bibitem{hambrook-trans}
K.~{Hambrook}.
\newblock {Explicit Salem sets and applications to metrical Diophantine
  approximation}.
\newblock \url{https://arxiv.org/abs/1604.00411v1}.

\bibitem{hambrook-R2}
K.~Hambrook.
\newblock Explicit {S}alem sets in {$\mathbb{R}^2$}.
\newblock {\em Adv. Math.}, 311:634--648, 2017.

\bibitem{jarnik-1}
V.~Jarn\'{\i }k.
\newblock {D}iophantischen {A}pproximationen und {H}ausdorffsches {M}ass.
\newblock {\em Mat. Sborjnik}, 36:371--382, 1929.

\bibitem{jarnik-2}
V.~Jarn\'{\i }k.
\newblock \"uber die simultanen diophantischen {A}pproximationen.
\newblock {\em Math. Z.}, 33(1):505--543, 1931.

\bibitem{kahane-book}
J.-P. Kahane.
\newblock {\em Some random series of functions}, volume~5 of {\em Cambridge
  Studies in Advanced Mathematics}.
\newblock Cambridge University Press, Cambridge, second edition, 1985.

\bibitem{kaufman}
R.~Kaufman.
\newblock On the theorem of {J}arn\'\i k and {B}esicovitch.
\newblock {\em Acta Arith.}, 39(3):265--267, 1981.

\bibitem{LP2009}
I.~{\L}aba and M.~Pramanik.
\newblock Arithmetic progressions in sets of fractional dimension.
\newblock {\em Geom. Funct. Anal.}, 19(2):429--456, 2009.

\bibitem{mattila-book-1}
P.~Mattila.
\newblock {\em Geometry of sets and measures in {E}uclidean spaces}, volume~44
  of {\em Cambridge Studies in Advanced Mathematics}.
\newblock Cambridge University Press, Cambridge, 1995.
\newblock Fractals and rectifiability.

\bibitem{mattila-book-2}
P.~Mattila.
\newblock {\em Fourier analysis and {H}ausdorff dimension}, volume 150 of {\em
  Cambridge Studies in Advanced Mathematics}.
\newblock Cambridge University Press, Cambridge, 2015.

\bibitem{melnicuk}
J.~V. Melni{\v{c}}uk.
\newblock Hausdorff dimension in {D}iophantine approximations of {$p$}-adic
  numbers.
\newblock {\em Ukrain. Mat. Zh.}, 32(1):118--124, 144, 1980.

\bibitem{mitsis}
T.~Mitsis.
\newblock A {S}tein-{T}omas restriction theorem for general measures.
\newblock {\em Publ. Math. Debrecen}, 60(1-2):89--99, 2002.

\bibitem{mock-thesis}
G.~Mockenhaupt.
\newblock Bounds in {L}ebesgue spaces of oscillatory integral operators.
\newblock Habilitationsschrift, Gesamthochschule Siegen, 1996.

\bibitem{mock}
G.~Mockenhaupt.
\newblock Salem sets and restriction properties of {F}ourier transforms.
\newblock {\em Geom. Funct. Anal.}, 10(6):1579--1587, 2000.

\bibitem{mock-ricker}
G.~Mockenhaupt and W.~J. Ricker.
\newblock Optimal extension of the {H}ausdorff-{Y}oung inequality.
\newblock {\em J. Reine Angew. Math.}, 620:195--211, 2008.

\bibitem{pa-thesis}
C.~Papadimitropoulos.
\newblock {\em The Fourier restriction phenomenon in thin sets}.
\newblock PhD thesis, University of Edinburgh, 2010.

\bibitem{pa-local}
C.~Papadimitropoulos.
\newblock Salem sets in local fields, the {F}ourier restriction phenomenon and
  the {H}ausdorff-{Y}oung inequality.
\newblock {\em J. Funct. Anal.}, 259(1):1--27, 2010.

\bibitem{pa-p}
C.~Papadimitropoulos.
\newblock Salem sets in the {$p$}-adics, the {F}ourier restriction phenomenon
  and optimal extension of the {H}ausdorff-{Y}oung inequality.
\newblock In {\em Vector measures, integration and related topics}, volume 201
  of {\em Oper. Theory Adv. Appl.}, pages 327--338. Birkh\"auser Verlag, Basel,
  2010.

\bibitem{salem}
R.~Salem.
\newblock On singular monotonic functions whose spectrum has a given
  {H}ausdorff dimension.
\newblock {\em Ark. Mat.}, 1:353--365, 1951.

\bibitem{shmerkin-suomala}
P.~Shmerkin and V.~Suomala.
\newblock Spatially independent martingales, intersections, and applications.
\newblock To appear in Memoirs of the Amer. Math. Soc.
\newblock \url{https://arxiv.org/abs/1409.6707v4}.

\bibitem{stein-book}
E.~M. Stein.
\newblock {\em Harmonic analysis: real-variable methods, orthogonality, and
  oscillatory integrals}, volume~43 of {\em Princeton Mathematical Series}.
\newblock Princeton University Press, Princeton, NJ, 1993.
\newblock With the assistance of Timothy S. Murphy, Monographs in Harmonic
  Analysis, III.

\bibitem{T75}
M.~H. Taibleson.
\newblock {\em Fourier analysis on local fields}.
\newblock Princeton University Press, Princeton, N.J.; University of Tokyo
  Press, Tokyo, 1975.
\end{thebibliography}

\noindent Robert Fraser \\ 
\textit{Address:} Department of Mathematics, University of British Columbia, Vancouver, BC V6T1Z2, Canada \\
%\curraddr{}
\textit{E-mail:} \texttt{rgf@math.ubc.ca} 

\vspace{10pt}

\noindent Kyle Hambrook \\
\textit{Address:} Department of Mathematics, University of Rochester, Rochester, NY 14627, USA \\
%\curraddr{}
\textit{E-mail:} \texttt{khambroo@ur.rochester.edu}
%\thanks{This work was supported by NSERC.}

\label{lastpage}
\end{document}